\documentclass[12pt]{article}
\usepackage[amsmath,cleveref]{macros}
\usepackage{enumerate}
\usepackage{mathtools}
\usepackage{url}
\usepackage{bbm}

\usepackage{graphicx}

\theoremstyle{definition}
\newtheorem{theoremdefinition}[theorem]{Theorem/Definition}

\DeclareMathOperator{\alg}{HSOP}

\DeclareMathOperator{\baj}{baj}
\DeclareMathOperator{\CGF}{CGF}

\DeclareMathOperator{\Gr}{Gr}
\DeclareMathOperator{\Hilb}{Hilb}

\DeclareMathOperator{\inv}{inv}

\DeclareMathOperator{\maj}{maj}
\DeclareMathOperator{\median}{median}

\let\Re\ralax\DeclareMathOperator{\Re}{Re}

\DeclareMathOperator{\sinc}{sinc}

\DeclareMathOperator{\Span}{Span}

\DeclareMathOperator{\uni}{uni}
\DeclareMathOperator{\lcc}{lcc}
\DeclareMathOperator{\gale}{Gale}

\newcommand{\bC}{\mathbb{C}}
\newcommand{\bE}{\mathbb{E}}

\newcommand{\bk}{\mathbbm{k}}
\newcommand{\bN}{\mathbb{N}}
\newcommand{\bP}{\mathbb{P}}
\newcommand{\bQ}{\mathbb{Q}}
\newcommand{\bR}{\mathbb{R}}
\newcommand{\bZ}{\mathbb{Z}}
\newcommand{\cA}{\mathcal{A}}
\newcommand{\cB}{\mathcal{B}}
\newcommand{\cC}{\mathcal{C}}

\newcommand{\cM}{\mathcal{M}}
\newcommand{\cN}{\mathcal{N}}

\newcommand{\cS}{\mathcal{S}}
\newcommand{\cU}{\mathcal{U}}
\newcommand{\cX}{\mathcal{X}}
\newcommand{\cY}{\mathcal{Y}}
\newcommand{\cZ}{\mathcal{Z}}

\numberwithin{equation}{section}


\title{Cyclotomic Generating Functions}

\author{Sara C. Billey\authornote{1}
\and
Joshua P. Swanson\authornote{2}
}

\authortext{1}{Department of Mathematics, University of Washington, Seattle, WA, USA
   (\email{billey@uw.edu}).}
\authortext{2}{Department of Mathematics, University of Southern California, Los Angeles, CA, USA (\email{swansonj@usc.edu}).}

\begin{document}

\maketitle

\begin{abstract}
  It is a remarkable fact that for many statistics on finite sets of
combinatorial objects, the roots of the corresponding generating
function are each either a complex root of unity or zero. These and
related polynomials have been studied for many years by a variety of
authors from the fields of combinatorics, representation theory,
probability, number theory, and commutative algebra. We call such
polynomials \textbf{cyclotomic generating functions} (CGFs). With
Konvalinka, we have studied the support and asymptotic distribution of the coefficients of several
families of CGFs arising from tableau and forest combinatorics. In
this paper, we survey general CGFs from algebraic, analytic, and
asymptotic perspectives. We review some of the many known examples of
CGFs in combinatorial representation theory; describe their
coefficients, moments, cumulants, and characteristic functions; and
give a variety of necessary and sufficient conditions for their
existence arising from probability, commutative algebra, and invariant
theory. As a sample result, we show that CGFs are ``generically''
asymptotically normal, generalizing a result of Diaconis on
$q$-binomial coefficients using work of Hwang--Zacharovas. We include
several open problems concerning CGFs.
\end{abstract}

\section{Introduction}\label{sec:intro}

Many formulas in enumerative combinatorics express the cardinality of
a finite set $X$ as a product or quotient of integers. In many cases
of interest, such formulas may be generalized to $q$-analogues with a
corresponding refined count of $X$ subject to the value of some
statistic on $X$.  Such $q$-analogues have nonnegative integer
coefficients.  A classic example is the \textbf{$q$-binomial
coefficients}, which $q$-count integer partitions that fit in the
$k\times (n-k)$ rectangle according to size for a given pair of
nonnegative integers $k\leq n$.  Here a partition
$\lambda =(\lambda_{1}\geq \cdots \geq \lambda_{j}\geq 0)$ fits in the
rectangle provided $j\leq k$ and each $\lambda_{i}$ is a nonnegative
integer less than or equal to $(n-k)$.  The size of $\lambda$ is the
sum of its parts, $|\lambda |=\sum \lambda_{i}$. These polynomials can be expressed
in two ways,
\begin{equation}\label{eq:qbinom}
  \binom{n}{k}_q \coloneqq \frac{[n]_q!}{[k]_q! [n-k]_q!} =
\sum_{\lambda \subset k\times (n-k)} q^{|\lambda |},
\end{equation}
where $[n]_q! \coloneqq [n]_q [n-1]_q \cdots [1]_q$ is a
\textbf{$q$-factorial} and $[n]_q \coloneqq 1 + q + \cdots + q^{n-1} =
(1-q^n)/(1-q)$ is a \textbf{$q$-integer}.  The $q$-binomial coefficient
$\binom{n}{k}_q$ is well known in topology
and geometry as the Poincar\'e polynomial of the Grassmannian variety
$\Gr(k, n)$ \cite{Fulton-book}. Here $X$ may be taken to be the set of Schubert cells of
$\Gr(k, n)$ and the $q$-statistic is the dimension of the cell.
Enumerative properties of the coefficients of $\binom{n}{k}_q$ have attracted
attention since the mid 19th century \cite{doi:10.1080/14786447808639408}.
They are a prime example of unimodality
\cite{Ohara.1990,Zeilberger.1986}.

In this paper, we study the class of such nonzero polynomials with
nonnegative integer coefficients arising as quotients of $q$-integers
from algebraic, analytic, and asymptotic perspectives.  Such
polynomials are closely associated with the cyclotomic polynomials
$\Phi_{n}$ from number theory. See \Cref{sec:background} for a review
of their key properties.  Since our examples all come from
$q$-counting formulas for well-known combinatorial objects and their
associated generating functions, we have chosen to call these
polynomials \textbf{cyclotomic generating functions} or \textbf{CGFs}
for short.  We begin by stating several equivalent characterizations
of this class. See \Cref{sec:algebraic} for the proof of their
equivalence.

\begin{theoremdefinition}\label{thmdef:cgf_equivalences}
Suppose $f(q)=\sum_{k=0}^{n} c_{k}q^{k}$ is a nonzero polynomial with
nonnegative integer coefficients.  The following are equivalent
definitions for the polynomial $f$ to be a \textbf{cyclotomic
generating function (CGF)}.
\medskip
  \begin{enumerate}[(i)]
    \item (Complex form.) The complex roots of $f(q)$ are
      all either roots of unity or zero.
\medskip

    \item (Kronecker form.) The complex roots of $f(q)$ all have
modulus at most 1 and the leading coefficient is the greatest common
divisor of all coefficients of $f$.
\medskip

    \item \label{item:cyclotomic} (Cyclotomic form.) The polynomial $f$ can be written
      as a positive integer times a product
      of cyclotomic polynomials and factors of $q$.
\medskip

    \item (Rational form.) There are multisets $\{a_1, \ldots, a_m\}$
      and $\{b_1, \ldots, b_m\}$
      of positive integers and integers $\alpha \in \bZ_{>0}$,
      $\beta \in \bZ_{\geq 0}$ such that
\[
        f(q) = \alpha q^\beta \cdot \prod_{j=1}^m
\frac{[a_j]_q}{[b_j]_q}
               = \alpha q^\beta \cdot \prod_{j=1}^m
                  \frac{1-q^{a_j}}{1-q^{b_j}}.
\]
      Moreover, this factorization is unique if the multisets are
      disjoint.  
\medskip

    \item (Probabilistic form.) There is a discrete random variable
      $\cX$ with probability distribution given by $P(\cX = k) = c_{k}/f(1)$ such that the following equality in distribution holds:
\[
        \cX + \cU_{b_1} + \cdots + \cU_{b_m}
	  = \beta + \cU_{a_1} + \cdots + \cU_{a_m},
\]
where the summands are all independent and $\beta$ is maximal such
that $q^{\beta} \mid f(q)$.  Here $\cU_a$ is a uniform random variable
supported on $\{0, 1, \ldots, a-1\}$ and $\alpha$ is the greatest common divisor of
all coefficients of $f$.  \medskip

    \item(Characteristic form.) There is a discrete random variable
$\cX$ on a uniform sample space supported on $\bZ_{\geq 0}$ with
probability generating function $\bE[q^{\cX}] = f(q)/f(1)$ such that
the scaled characteristic function $\phi_{\cX} (z):=\bE[e^{2\pi i z\cX}]$ has complex
zeros only at rational $z \in \bQ$.
\end{enumerate}
\end{theoremdefinition}

\begin{example}
Allow us to indulge in a thought experiment. Imagine yourself as Percy MacMahon at the turn of
the last century investigating \textbf{plane partitions} in an $x \times y \times z$
box for the first time. These are $x \times y$ matrices of nonnegative integers whose entries are at
most $z$ and weakly decrease along rows and columns. For $(x, y, z) = (3, 2, 2)$, you find there are $50$
plane partitions. We have $50 = 2 \cdot 5^2$, and a wide-eyed optimist may hope for a product formula akin
to that for $\binom{n}{k}$, though cancellations are difficult to uncover. Plane partitions come with the
\textbf{size} statistic, which is the sum of all entries. The size generating function here is
  \[ f_{(3,2,2)}(q)= 1 + q + 3q^2 + 4q^3 + 6q^4 + 6q^5 + 8q^6 + 6q^7 + 6q^8 + 4q^9 + 3q^{10} + q^{11} + q^{12}. \]
You immediately notice the generating function is monic, palindromic, and even unimodal. Ever-optimistic,
you try dividing off cyclotomic polynomial factors, which results in $f_{(3,2,2)}(q)=\Phi_6(q)\, \Phi_5(q)^2\, \Phi_4(q)$---you've found a CGF! Unique factorization as a reduced quotient of $q$-integers results in
  \[ f_{(3,2,2)}(q)=\frac{[6]_q\, [5]_q^2\, [4]_q}{[3]_q\, [2]_q^2\, [1]_q}. \]
We now see that most of the cancellations are ``hidden'' in the $q=1$ specialization $2 \cdot 5^2$.
It is not hard to imagine that further experimentation from here quickly leads you to the following version of
MacMahon's famous formula,
\begin{equation}\label{eq:PP}
f_{(x,y,z)}(q)=  \prod_{i=1}^x \prod_{j=1}^y \frac{[i+j+z-1]_q}{[i+j-1]_q},
\end{equation}
which for $(x, y, z) = (3, 2, 2)$ yields
  \[ f_{(3,2,2)}(q)= \frac{[6]_q [5]_q [4]_q [5]_q [4]_q [3]_q}{[4]_q [3]_q [2]_q [3]_q [2]_q [1]_q}. \]
\end{example}

\begin{definition}\label{def:basic}
  Among all CGFs, we focus on the family of \textbf{basic CGFs} with no $\alpha$ or $q^{\beta}$
factors, 
\begin{equation}\label{eq:basic_CGF}
        f(q) =  \prod_{j=1}^m \frac{[a_j]_q}{[b_j]_q}
	      = \prod_{j=1}^m  \frac{1-q^{a_j}}{1-q^{b_j}}.
\end{equation}
  The \textbf{basic CGF monoid} $\Phi^+$ is the monoid consisting of all basic CGFs under multiplication.
\end{definition}

By the rational form, basic CGFs are always palindromic and monic. They need not be unimodal in general, e.g.
\begin{equation}\label{eq:example1}
f(q)=q^{6} + q^{4} + q^{3} + q^{2} + 1 =
\frac{[5]_{q}[6]_{q}}{[2]_{q}[3]_{q}} \in \Phi^{+}.
\end{equation}
In \Cref{sec:monoid}, we study the submonoids $\Phi^{\uni}$ and
$\Phi^{\lcc}$ of $\Phi^+$ given by cyclotomic generating functions which are
unimodal or log-concave with no internal zeros, respectively.

Testing if a polynomial $f \in \mathbb{Z}_{\geq 0}[q]$ is in $\Phi^+$ may be done by
repeatedly dividing off cyclotomic polynomial factors $\Phi_n(q)$. Note that
we have the elementary bounds $\sqrt{n/2} \leq \deg \Phi_n(q) \leq n$, so there are only a
finite number of basic CGFs of each given degree.

The main reason to study CGFs as a family of polynomials is that they
are already prevalent in the literature.  In addition to $q$-binomial
coefficients, the standard $q$-analogues of $n!$ using inversions or
the major index statistic on permutations and their generalizations to
arbitrary words corresponding to $q$-multinomial generating
functions are CGFs using MacMahon's classic formulas.  Here are a
(Catalan) number of examples from the literature. We emphasize that this
list is not exhaustive.

\begin{enumerate}
\item Length generating functions of Weyl groups and their parabolic
quotients \cite{b-b}.  
\item The Hilbert series of all finite-dimensional quotient rings of
the form
  \[ R \coloneqq B/(\theta_1, \ldots, \theta_m) \]
where $\theta_1, \ldots, \theta_m$ is a homogeneous system of parameters in the
polynomial ring $B = \bk[x_1, \ldots, x_m]$, where $\deg(\theta_i) = a_i$
and $\deg(x_i) = b_i$.  The ring of invariants of a finite reflection group has
such a Hilbert series.   See \Cref{sub:hsop} for references and
further discussion.  
\item A $q$-analogue of Cayley's formula coming from the $t=q$ case of the diagonal
harmonics, namely $q^{\binom{n}{2}} \Hilb(DH_n; q, q^{-1}) =
[n+1]_q^{n-1}$ \cite[Thm. 4.2.4]{Haiman.2003}. 

\item A $q$-analogue of the Catalan numbers \cite{Chen-Wang-Wang.2008}.   
\item A $q$-analogue of the Fu{\ss}--Catalan numbers for irreducible well-generated complex reflection groups \cite[(3.2),~Thm.~25]{MR3363972}.
\item A $q$-analogue of Narayana numbers \cite[Thm. 1.10]{MR3884758}.  
\item A $q$-analogue of the Hook Length Formula for standard Young tableaux \cite[Prop.~4.11]{stanley.1979},
\item A $q$-analogue of the Hook Content Formula for semi-standard
Young tableaux \cite[Thm.~7.21.2]{ec2}.
\item A $q$-hook length formula for linear extensions
of forests \cite{MR1022316}.

\item A $q$-analogue of the Weyl dimension formula for highest weight
modules of semisimple Lie algebras \cite[Lem.~2.5]{MR1262215}.

\item The $q$-analogues of the formulas enumerating alternating sign
matrices \cite[p.171]{MR3156682}, cyclically symmetric plane partitions
\cite{MR652647}, or totally symmetric plane partitions
\cite{MR2775659}.

\item Rank generating functions of Bruhat intervals $[\mathrm{id},w]$ for
permutations $w$ indexing smooth Schubert varieties in the complete
flag manifolds \cite{gasharov97,GR2000}.  Similar results hold in
other types as well.  See \cite{MR3406513} for an extensive overview
and recent results.

\item The statistic $\baj - \inv$ appeared in the context of extended
affine Weyl groups and Hecke algebras in the work of Iwahori and
Matsumoto in 1965 \cite{Iw.M}.  It is the Coxeter length generating function
restricted to coset representatives of the extended affine Weyl group
of type $A_{n-1}$ mod translations by coroots.  Stembridge and Waugh
\cite[Remarks 1.5 and 2.3]{MR1692145} give a careful overview of this
topic and further results.  In particular, they prove the
corresponding $q$-analogue of $n!$ is a cyclotomic generating function.
See Zabrocki \cite{math/0310301} for the nomenclature and
\cite[\S4]{BKS-asymptotics} for an asymptotic description.

\item Rank generating functions of Gaussian posets and $d$-complete
posets, with connections to Lie theory and order polynomials of
posets, \cite{MR782055,MR4038788,MR1262215}.  In recent work, Hopkins
has explored general properties of posets with order polynomial
product formulas in the context of cyclic sieving and other good
dynamical behavior such as promotion and rowmotion
\cite{Hopkins.2023}.  Furthermore, sometimes there is an associated
cyclic sieving phenomenon associated to these order polynomials.  These
well behaved order polynomials are often cyclotomic generating
functions.  See also Stanton's online notes \cite{Stanton.note} for
conjectures related to CGFs generalizing rank generating functions of
Gaussian posets, which he calls ``fake Gaussian sequences''.

\end{enumerate}

Classes of polynomials similar to the class of cyclotomic generating
functions have been studied for roughly a century in other contexts. A
polynomial $f \in \bC[q]$ of degree $n$ is \textbf{self-inversive} if
$f(q) = \alpha q^n f(q^{-1})$ where $|\alpha| = 1$. Equivalently, the
zeros of $f$ are symmetric about the unit circle. Self-inversive
polynomials were studied by A.~Cohn \cite{Cohn22} a century ago and by
Bonsall--Marden \cite{BonsallMorris} in the 1950s; they related them
to the complex roots of $f'$.

Kedlaya \cite{Kedlaya} and Hwang--Zacharovas \cite{HwangZacharovas}
refer to polynomials with roots all on the unit circle as
\textbf{root-unitary}. Note that basic CGFs are in particular
root-unitary with real coefficients.  We will review key details from
\cite{HwangZacharovas} in \Cref{sec:background} since they pertain to
our focus on basic CGFs.  In particular, they gave an elegant formula
for the cumulants of the real-valued discrete random variables whose
probability generating functions are basic CGFs and a test for
asymptotic normality using the fourth cumulants.  More recently, there
has been interest in certain polynomials associated to numerical
semigroups that have cyclotomic factorizations
\cite{MR3483156,MR4374779} and identifying criteria for central limit
theorems for polynomials based on the geometry of their roots in the
complex plane \cite{heerten2023probabilistic,Michelen-Sahasrabudhe.2019,michelen2019central}.

Another particularly well-known class of polynomials consists of
elements of $\bR_{\geq 0}[q]$ with all real roots \cite{Branden2015,Brenti1994,Pitman1997,MR1110850}. Stirling numbers of
the second kind are the motivating example \cite{Harper}. A famous
asymptotic characterization of their coefficients is due to Bender.

\begin{theorem}[{\cite[Thm.~2]{Bender}; see also \cite{Harper}}]\label{thm:bender}
  If $\cX_n$ is a sequence of real-valued discrete random variables
  whose probability generating functions are polynomials $f_n(q)$
  with all real roots and standard deviation $\sigma_n \to \infty$,
  then $\cX_n$ is asymptotically normal.
\end{theorem}

In contrast to \Cref{thm:bender}, CGF polynomials and their associated
random variables have much more complex limiting behavior.  In
\cite{BS-asymptotics}, we initiated the study of the metric space of
all standardized CGF distributions in the L\'evy metric and their
asymptotic limits.  For some families of CGFs, there are statistics
that completely characterize their limiting behavior.  Asymptotic
normality is one common occurrence.  While finding a complete
description of the closure of all CGF distributions is still an open
problem \cite[Open Problem 1.19]{BS-asymptotics}, we did show that all
uncountably many \textbf{DUSTPAN distributions} (\textbf{d}istributions
associated to a \textbf{u}niform \textbf{s}um for $\mathbf{t}$ \textbf{p}lus \textbf{a} \textbf{n}ormal distribution) can
occur on the boundary.  Furthermore, many more multimodal limiting
distributions are possible, due to the following construction. Given a
CGF $f(q)$ and positive integer $N$, note that $(1+q^N) f(q)$ remains
a CGF. If $N$ is larger than the degree of $f(q)$, the result is a CGF
with two disjoint copies of the distribution of $f(q)$. In this way,
valid limits may be ``duplicated'' in a fractal pattern. The unimodal
and log-concave submonoids of $\Phi^{+}$ are not subject to this
construction.  We expect the unimodal basic CGFs will have the
simplest limiting behavior.

Towards addressing the problem of characterizing when a limiting
sequence of standardized CGF distributions in the L\'evy metric
approaches the standard normal distribution $\cN(0,1)$, we present a new
criterion in terms of the multisets in the rational presentation of
the CGF polynomials.  It is these multisets which have combinatorial
significance in applications.  We give a more complete
characterization in the special case of polynomials which are products
of $q$-integers, see \Cref{thm:q_prod_an}.  This approach builds on
the work of Hwang--Zacharovas \cite{HwangZacharovas}, but emphasizes
the role of the multisets of $q$-integers in the numerators.  It also
generalizes the work of Diaconis \cite[pp.128-129]{MR964069} who
showed that the coefficients of a sequence of $q$-binomials
$\binom{n}{k}_q$ are asymptotically normal provided both $k, n-k \to
\infty$.  See \Cref{sec:asymcgfs} on asymptotic considerations for the
proof.

\begin{theorem}\label{thm:diaconis}
  Let $f_N(q) = \prod_{a \in a^{(N)}} [a]_q/\prod_{b \in b^{(N)}}
[b]_q \in \Phi^{+}$ for $N=1,2,\dots$ be a sequence of basic
cyclotomic generating functions expressed in terms of their associated multisets
$a^{(N)}, b^{(N)}$.  For each $N$, let $\cX_N$ be the corresponding
CGF random variable with $\bE[q^{\cX_N}] = f_N(q)/f_N(1)$. If
  \begin{equation}\label{eq:diaconis.1}
\limsup_{N \to \infty} \frac{\sum_{b \in b^{(N)}} (b^2 - 1)} {\sum_{a
\in a^{(N)}} (a^2 - 1)} < 1 \end{equation}
and \begin{equation}\label{eq:diaconis.1b} \lim_{N \to \infty} \sum_{a
\in a^{(N)}} \left(\frac{a}{\max a^{(N)}}\right)^{4}  = \infty,
\end{equation}
then $\cX_1, \cX_2, \ldots$ is asymptotically normal.
\end{theorem}

The ratios in \eqref{eq:diaconis.1} are related to the variances of
the random variables obtained from sums of uniform random variables
associated to the numerators and denominators respectively.  The sum
in \eqref{eq:diaconis.1b} is derived using the fourth cumulant test
due to Hwang--Zacharovas.  Hence the appearance of the quadratic and
quartic terms. When the limiting ratio in \eqref{eq:diaconis.1} is $1$,
massive cancellation occurs and obscures the asymptotic behavior.
More careful analysis may still reveal limit laws, such as in the
\textbf{Irwin--Hall} case of \cite[Thm.~1.7]{BKS-asymptotics}.

\begin{example}
  Continuing the plane partition example, suppose we have an infinite
sequence of triples $(x_{N}, y_{N}, z_{N})$ for $N=1,2,\ldots$. For
ease of notation, we will just write $(x, y, z)$ for a general element
in the sequence.    In MacMahon's formula \eqref{eq:PP}, we see
intuitively that the numerator dominates the denominator, which is
essentially condition \eqref{eq:diaconis.1}.  More rigorously, routine
calculations show that the ratio \eqref{eq:diaconis.1} is the rational
function \[ \frac{g(x, y)}{g(x, y)+z \cdot (x+y+z)}, \qquad\text{where
}g(x, y) = (2x^2+3xy+2y^2-7)/6. \] When $x \leq y \leq z$, which may
be arranged without loss of generality, and $z \to \infty$, we use the
rational function above to observe \eqref{eq:diaconis.1} holds. Similarly we find that
\eqref{eq:diaconis.1b} holds if $y \to \infty$ as well. Hence with
nothing more than basic computer algebra systems and routine
calculations, we find that size on plane partitions in a box is
asymptotically normal if $\mathrm{median}(x, y, z) \to \infty$,
recovering a result from \cite{BS-asymptotics}.  Indeed, this result
is sharp in the sense that, when $\mathrm{median}(x, y, z)$ is bounded
and $z \to \infty$, the possible limits are Irwin--Hall
distributions; see \cite{BS-asymptotics}.  The approach in this example
to proving central limit theorems is a kindred spirit to Zeilberger's
``automatic central limit theorem generator'' \cite{MR2683231}.
\end{example}

\begin{remark}
  While \Cref{thm:diaconis} provides a precise asymptotic conclusion, \cite[Thm.~1.4]{michelen2019central} states the following quantitative bound after a lengthy technical argument. If $f(q)$ is a polynomial which is zero-free in the sector $\{z \in \mathbb{C} : |\arg(q)| < \delta\}$ and if $\cX$ is a random variable with $\bE[q^{\cX}] = f(q)/f(1)$, then 
  \begin{equation}\label{eq:MS.1}
    \sup_{t \in \bR} |\bP(\cX^* \leq t) - \bP(\cZ \leq t)| \leq \frac{C}{\delta \sigma},
  \end{equation}
  where $C$ is an absolute constant and $\cZ$ is the standard normal random variable. In the situation of \Cref{thm:diaconis}, $\delta^{-1} = \max a^{(N)}/2\pi$ and $12\sigma^2 = \sum_{a \in a^{(N)}} a^2 - \sum_{b \in b^{(N)}} b^2$, so \eqref{eq:MS.1} becomes
  \begin{equation}\label{eq:MS.2}
    \sup_{t \in \bR} |\bP(\cX^* \leq t) - \bP(\bZ \leq t)| \leq \frac{C_0 \max a^{(N)}}{\sqrt{\sum_{a \in a^{(N)}} a^2 - \sum_{b \in b^{(N)}} b^2}}.
  \end{equation}
  Using MacMahon's plane partition formula \eqref{eq:PP}, $\max a^{(N)} = x+y+z-1$ and $12\sigma^2 = xyz(x+y+z)$. The right-hand side in \eqref{eq:MS.2} in this case may then be replaced with $C_0 \sqrt{\frac{x+y+z}{xyz}}$. Note that this expression tends to $0$ if and only if $\median(x, y, z) \to \infty$.
\end{remark}

Towards addressing the problem of characterizing all limiting
distributions for CGF random variables, we note the following special property. Roughly speaking, it states that the converse of the Frech\'et--Shohat Theorem holds for standardized CGF distributions. A proof is given in \Cref{sec:analytic}.  The same statement holds with moments replaced by cumulants, and in fact, this paper focuses on the cumulants as the key quantities characterizing a CGF distribution.  

\begin{theorem}\label{thm:FS_iff}
  Let $\cX_1, \cX_2, \ldots$ be a sequence of random variables
  corresponding to cyclotomic generating functions. Then the sequence
  of standardized random variables $\cX_1^*, \cX_2^*, \ldots$
  converges in distribution if and only if for all $d \in \bZ_{\geq
  1}$, the limit of the $d^{th}$ moments of the random variables  
    \[ \mu_d := \lim_{N \to \infty} \mu_d^{\cX_N^*} \]
exists and is finite. In this case, $\cX_n^* \Rightarrow \cX$ where
$\mu_d = \mu_d^{\cX}$ is the $d^{th}$ moment of $\cX$, $\cX$ is determined by its moments, and the moment-generating function of $\cX$ is entire.
\end{theorem}

The paper is organized as follows.  In \Cref{sec:background}, we state
our notation and give background details from prior work.
In \Cref{sec:algebraic,sec:asymcgfs,sec:analytic}, we describe the algebraic, asymptotic and
analytic properties of basic CGFs, their associated random variables,
and characteristic functions.  In \Cref{sec:examples}, we examine the
monoid of basic cyclotomic generating functions and several of its
submonoids including unimodal and log-concave CGFs, CGFs whose
rational form satisfies the Gale order, and CGFs that come from
Hilbert series of polynomial rings modded out by homogeneous systems of
parameters.  We include several open problems about cyclotomic
generating functions for future work based on experimentation.

\section{Background}\label{sec:background}

\subsection{Cyclotomic polynomials}

The \textbf{cyclotomic polynomials} are defined for all positive
integers $n\geq 1$ by
\begin{equation}\label{eq:1}
\Phi_{n}(q) = \prod_{\substack{1\leq k\leq n\\ \gcd(k,n)=1}}  (q - e^{2\pi i k/n}).
\end{equation}
Each $n^{th}$ root of unity $e^{2\pi i k/n}$ is in the cyclic group generated by
$e^{2\pi i/d}$ for $d \in \{1,2,\dots ,n \}$ where $d=n/\gcd(k,n)$.  Thus, for $n\geq
1$, we have 
\begin{equation}\label{eq:2}
q^{n}-1 = \prod_{k=1}^{n} (q - e^{2\pi i k/n}) = \prod_{d |n}
\Phi_{d}(q).  
\end{equation}
Since $q^{n}-1=(q-1)(1+q+q^{2}+ \cdots + q^{n-1})$, we have
\begin{equation}\label{eq:q-integers}
[n]_{q} = 1+q+q^{2}+ \cdots + q^{n-1} = \prod_{1<d |n} \Phi_{d}(q).
\end{equation}

Let $\mu(n)$ be the classical M\"obius function.  The M\"obius
function satisfies the recurrence $\sum_{d \mid n} \mu (n/d)=0$ for
all $n>1$.  Therefore, by M\"obius inversion, we also have the
identities for cyclotomic polynomials indexed by $n>1$,
\begin{equation}\label{eq:mobius.inversion}
\Phi_n(q) = \prod_{d \mid n} (q^d - 1)^{\mu(n/d)} = \prod_{d \mid n}
[d]_q^{\mu(n/d)}.
\end{equation}
It follows by induction that $\Phi_n(q)$ is monic with integer
coefficients.

The cyclotomic polynomials in the product \eqref{eq:2} are the
irreducible factors of $(q^{n}-1)$ over the integers. For example,
$\Phi_{1}(q)=q-1$, $\Phi_{2}(q)=q+1$, $\Phi_{3}(q)=1+q+q^{2}$,
$\Phi_{4}(q)=q^{2}+1$, and $\Phi_{27}(q)=q^{18}+q^{9}+1$.  See
\cite{wiki:cyclotomic} or \cite[Sect. 7.7]{Cohn} for many beautiful
properties of cyclotomic polynomials including the formula
\[
\Phi_{p^{k}}(q)= \sum_{j=0}^{p-1} q^{j\, p^{k-1}}.
\]
Also, $\Phi_{p}(q)=1+q + q^{2}+\cdots + q^{p-1}=[p]_{q}$ if and only if
$p$ is prime.  Note, $\Phi_{p^{k}}(1)=p$, otherwise if $n>1$ is not a
prime power, then $\Phi_{n}(1)=1$.  The number of complex roots and
therefore the degree of $\Phi_{n}(q)$ is given by Euler's totient
function, $\varphi(n)$, so for $n\geq 3$ the degree of $\Phi_{n}(q)$
is even and bounded by $\sqrt{n/2}\leq \varphi(n)\leq n-1$.  For
$n\geq 2$, the constant term is $\Phi_{n}(0)=1$, and the coefficients are
\textbf{palindromic} in the sense that $\Phi_{n}(q) = q^{\varphi (n)}
\Phi_{n}(q^{-1})$.

\subsection{Probabilistic generating functions and cumulants}\label{sub:prob.and.cum}

We now briefly review the background from probability related to
CGFs.  See \cite{MR1324786} for more details or
\cite[\S2]{BKS-asymptotics} for a review aimed at a combinatorial
audience.

The \textbf{probability generating function} of a discrete random
variable $\cX$ supported on $\bZ_{\geq 0}$ is
\[
G_{\cX}(q) \coloneqq \bE[q^{\cX}] = \sum_{k=0}^\infty P(\cX = k) q^k.
\]
If $\cX$ is a nonnegative-integer-valued statistic on a 
set $S$ which is sampled uniformly, then the probability generating
function of $\cX$ as a random variable is the same as the ordinary
generating function of $\cX$, up to a scale factor:
\[
G_{\cX}(q) = \frac{1}{\#S} \sum_{s \in S} q^{\cX(s)}.
\]
The \textbf{moment generating function} of $\cX$ is obtained by
substituting $q=e^{t}$ above so 
 \[ M_{\cX}(t) \coloneqq \bE[e^{t\cX}] = G_{\cX}(e^t) = \sum_{d=0}^\infty \mu_d \frac{t^d}{d!},
\]
where $\mu_d \coloneqq \bE[\cX^d]$ is the $d$th \textbf{moment}. The
\textbf{central moments} of $\cX$ are $\alpha_d \coloneqq \bE[(\cX -
\mu)^d]$, where $\mu = \mu_1$ is the mean and $\alpha_2 = \sigma^2$ is
the variance. The \textbf{characteristic function} of $\cX$ is
\[
\phi_{\cX}(t) \coloneqq \bE[e^{it\cX}] = M_{\cX}(it) = G_{\cX}(e^{it}).
\]
The characteristic function, for $t \in \bR$, exists for all random
variables and determines the distribution of $\cX$. While moment
generating functions in general are less well-behaved, all of the
moment generating functions we will encounter converge in a complex
neighborhood of $0$ and the distributions will be determined by their
moments.

The \textbf{cumulant generating function}, also known as the
\textbf{second characteristic function}, of $\cX$ is $\log
\phi_{\cX}(t) = \log M_{\cX}(it)$. The \textbf{cumulants} $\kappa_1,
\kappa_2, \ldots$ of $\cX$ are defined to be the coefficients of the
related exponential generating function \[ K_{\cX}(t) \coloneqq \log
M_{\cX}(t) = \log \bE[e^{t\cX}] = \sum_{d=1}^\infty \kappa_d
\frac{t^d}{d!}. \] Cumulants and moments are polynomials in each other
and are interchangeable for many purposes, though cumulants generally
have more convenient formal properties. For example, for independent
random variables $\cX$ and $\cY$, $\kappa_d^{\cX+\cY} = \kappa_d^{\cX}
+ \kappa_d^{\cY}$ for all positive integers $d$.

The normal distribution with mean $\mu$ and standard deviation
$\sigma$, denoted $\cN(\mu , \sigma )$ is central in this work.  It is
the unique distribution with cumumlants $\kappa_{1}=\mu,
\kappa_{2}=\sigma^{2}$, and $\kappa_{d}=0$ for all positive integers
$d\geq 3$.

The cumulants of random variables associated to cyclotomic generating
functions have the following simple, explicit form due to
Hwang--Zacharovas.  This builds on work of Chen--Wang--Wang
\cite[Thm. 3.1]{Chen-Wang-Wang.2008} and Sachkov
\cite[\S1.3.1]{MR1453118}. From their formula it is easy to derive the
formula for the moments and central moments of CGF random variables as
explained in \cite[\S2.3]{BKS-asymptotics}.

\begin{theorem}\cite[\S4.1]{HwangZacharovas}\label{thm:prod_cumulants}
  Suppose $\{a_1, \ldots, a_m\}$ and $\{b_1, \ldots, b_m\}$ are
  multisets of positive integers such that
  \[ f(q) = \frac{\prod_{k=1}^m [a_k]_q}{\prod_{k=1}^m [b_k]_q}
    =\sum_{k=0}^{n} c_k q^k
    \in
    \Phi^{+}, \]
  so in particular each $c_k \in \bZ_{\geq 0}$ and $n=\sum a_{k}-b_{k}\geq 0$.
  Let $\cX$ be a discrete random variable with
  $\bP[\cX=k]=c_k/f(1)$. Then the $d$th
  cumulant of $\cX$ is
  \begin{equation}\label{eq:prod_cumulants}
    \kappa_d^{\cX} =  \frac{B_d}{d} \sum_{k=1}^m (a_k^d - b_k^d),
  \end{equation}
  where $B_d$ is the $d$th Bernoulli number (with $B_1 =
  \frac{1}{2}$).
  Moreover, the $d$th central moment of $\cX$ is
  \begin{equation}\label{eq:prod_central_moments}
    \alpha_d = \sum_{\substack{|\lambda|= d \\ \text{has all parts even}}}
      \frac{d!}{z_\lambda}\, \prod_{i=1}^{\ell(\lambda)}
        \frac{B_{\lambda_i}}{\lambda_i!}
        \left[\sum_{k=1}^m \left(a_k^d - b_k^d\right)\right],
  \end{equation}
  and the $d$th moment of $\cX$ is
  \begin{equation}\label{eq:prod_moments}
    \mu_d = \sum_{\substack{|\lambda| = d \\ \text{has all parts either}
      \\ \text{even or size $1$}}}
      \frac{d!}{z_\lambda} \, \prod_{i=1}^{\ell(\lambda)}
        \frac{B_{\lambda_i}}{\lambda_i!}
        \left[\sum_{k=1}^m \left(a_k^d - b_k^d\right)\right].
  \end{equation}
\end{theorem}

\begin{definition}
  Let $\cX_1, \cX_2, \ldots$ and $\cX$ be real-valued random variables
  with cumulative distribution functions $F_1, F_2, \ldots$ and $F$,
  respectively. We say $\cX_1, \cX_2, \ldots$ \textbf{converges in distribution}
  to $\cX$, written $\cX_n \Rightarrow \cX$, if for all $t \in \bR$ at
  which $F$ is continuous we have
    \[ \lim_{n \to \infty} F_n(t) = F(t). \]
\end{definition}

For any real-valued random variable $\cX$ with mean $\mu$ and variance
$\sigma^2 > 0$, the corresponding \textbf{standardized random variable}
is \[ \cX^* \coloneqq \frac{\cX-\mu}{\sigma}. \] Observe that $\cX^*$
has mean $\mu^* = 0$ and variance ${\sigma^*}^2 = 1$. The moments and
central moments of $\cX^*$ agree for $d \geq 2$ and are given
by \[\mu_d^* = \alpha_d^* = \alpha_d/\sigma^d. \] Similarly, the
cumulants of $\cX^*$ are given by $\kappa_1^* = 0$, $\kappa_2^* = 1$,
and $\kappa_d^* = \kappa_d/\sigma^d$ for $d \geq 2$.

\begin{definition}\label{def:asymptotically.normal}
  Let $\cX_1, \cX_2, \ldots$ be a sequence of real-valued random variables.
  We say the sequence is \textbf{asymptotically normal} if
  $\cX_n^* \Rightarrow \cN(0, 1)$.
\end{definition}

We next describe two standard criteria for establishing
asymptotic normality or more generally convergence in
distribution of a sequence of random variables.

\begin{theorem}[L\'evy's Continuity Theorem, {\cite[Theorem~26.3]{MR1324786}}]\label{thm:levy}
  A sequence $\cX_1, \cX_2, \ldots$ of real-valued random variables
  converges in distribution to a real-valued random variable $\cX$
  if and only if, for all $t \in \bR$,
    \[ \lim_{n \to \infty} \bE[e^{it\cX_n}] = \bE[e^{it\cX}]. \]
\end{theorem}

\begin{theorem}[Frech\'et--Shohat Theorem,
  {\cite[Theorem~30.2]{MR1324786}}]\label{thm:moments}
  Let $\cX_1, \cX_2, \ldots$ be a sequence of real-valued random variables,
  and let $\cX$ be a real-valued random variable. Suppose
  the moments of $\cX_n$ and $\cX$ all exist and the moment
  generating functions all have positive radius of convergence. If
  \begin{equation}\label{eq:moments_criterion}
    \lim_{n \to \infty} \mu_d^{\cX_n} = \mu_d^{\cX} \hspace{.5cm} \forall
     d \in \bZ_{\geq 1},
  \end{equation}
  then $\cX_1, \cX_2, \ldots$ converges in distribution to $\cX$.
\end{theorem}

By \Cref{thm:levy}, we may test for asymptotic normality by
checking if the standardized characteristic functions tend pointwise to
the characteristic function of the standard normal. Likewise
by \Cref{thm:moments} we may instead perform the check on the
level of individual standardized moments, which is often referred to as the
\textbf{method of moments}. By the polynomial relationship between moments and
cumulants, we may further replace the moment condition
\eqref{eq:moments_criterion} with the cumulant condition
\begin{equation}\label{eq:cumulants_criterion}
  \lim_{n \to \infty} \kappa_d^{\cX_n} = \kappa_d^{\cX}.
\end{equation}
For instance, we have the following explicit criterion.  
  
\begin{corollary}\label{cor:cumulants}
  A sequence $\cX_1, \cX_2, \ldots$ of real-valued
  random variables on finite sets is asymptotically normal
  if for all $d \geq 3$ we have
  \begin{equation}\label{eq:cumulants_criterion2}
    \lim_{n \to \infty} \frac{\kappa^{\cX_n}_d}{(\sigma^{\cX_n})^d} = 0.
  \end{equation}
\end{corollary}

In fact, the converse of the Frech\'et--Shohat theorem holds for
cyclotomic generating functions. See
\Cref{thm:frechet_shohat_converse} below, which builds on \cite[Lem.~2.8]{HwangZacharovas}.  
Furthermore, we have the following simplified test for asymptotic normality due
to Hwang and Zacharovas.

\begin{theorem}\label{thm:4thcumulant.test}\cite[Thm. 1.1]{HwangZacharovas}
  Let $f_1(q), f_2(q), \ldots$ be a sequence of cyclotomic generating
functions. Let $\cX_1, \cX_2, \ldots$ be a corresponding sequence of
random variables with $\bE[q^{\cX_N}] = f_N(q)/f_N(1)$. Then, $\cX_1,
\cX_2, \ldots$ is asymptotically normal
if and only if the standardized fourth cumulants approach 0, 
  \begin{equation}\label{eq:cumulants_criterion3}
    \lim_{n \to \infty} \frac{\kappa^{\cX_n}_4}{(\sigma^{\cX_n})^4} = 0.
  \end{equation}
\end{theorem}

\subsection{Formal cumulants}\label{ssec:formal.cum}

We may extend the notions of cumulants and moments to power series
even when they do not necessarily have associated discrete random
variables.  Suppose that $f(q) \in R[[q]]$ is a formal power series
with coefficients in a (commutative, unital) ring $R$ of
characteristic $0$. If $f(1)=1$, one may define the \textbf{formal
cumulants} of $f$ by the coefficients in the expansion of the
generating function 
\begin{equation}
\log f(e^t) = \sum_{d=1}^\infty \kappa_d(f) \frac{t^d}{d!} 
\end{equation}
where $e^t \coloneqq \sum_{k=0}^\infty \frac{t^n}{n!} \in R[[t]]$.
See \cite{Pistone.Wynn.1999} or \cite[\S2]{BKS-asymptotics} for more
details.  If $f$ is clear from context, we will often just write
$\kappa_{d}$ for $ \kappa_d(f)$. Similarly, we write $\mu =
\kappa_{1}(f)$ and $\sigma^{2}=\kappa_{2}(f)$. If $f(1) \neq 0$ is invertible, we use
$\kappa_d(f) \coloneqq \kappa_d(f/f(1))$. If $R=\mathbb{C}$, we
may also define the \textbf{formal characteristic function} of $f$ by
$\phi_{f}(t) \coloneqq f(e^{it})/f(1)$.

For example, if $f(q) = \frac{\prod_{k=1}^m
[a_k]_q}{\prod_{k=1}^m [b_k]_q} \in \mathbb{Z}[[q]]$, the cumulant formula
\eqref{eq:prod_cumulants} remains valid, so
\begin{equation}\label{eq:formal.cum}
 \kappa_d(f) = \frac{B_d}{d} \left(\sum_{k=1}^m a_k^d - \sum_{k=1}^m
b_k^d\right). 
\end{equation}
We also have two factored forms for the formal characteristic function
\begin{equation}\label{eq:cgf_cf}
  \phi_{f}(t) \coloneqq f(e^{it})/f(1) = \frac{\prod_{k=1}^m
[a_k]_{e^{it}}/a_k} {\prod_{k=1}^m [b_k]_{e^{it}}/b_k} = e^{-it\mu}
\prod_{j=1}^m \frac{\sinc(a_j t/2)}{\sinc(b_j t/2)},
\end{equation}
where $[a]_{e^{it}}=(1+e^{it}+e^{2it}+ \cdots + e^{(a-1)it}),$\, $\sinc(x) \coloneqq \frac{\sin x}{x}$ and $\mu = \kappa_1(f)$.
This coincides with the actual characteristic function
$\phi_\cX(t)$ when $f(q)$ is a cyclotomic generating
function with corresponding random variable $\cX$.

In particular, the formal mean of the cyclotomic polynomial
$\Phi_{n}(q)$ is
$$
\mu = \kappa_{1}(\Phi_n(q))= B_1 \sum_{k \mid n} \mu(n/k) k
= B_{1}\ n\prod_{\substack{\text{$p$ prime} \\ p \mid n}} \left(1 -
\frac{1}{p}\right) =\varphi (n)/2
$$
for any $n>1$ by known formulas of the M\"obius function and Euler's
totient function $\varphi (n)$.  Thus, $\mu$ is half the degree of
$\Phi_n(q)$ as expected by \eqref{eq:1}. More generally, the formal
cumulants of cyclotomic polynomials are closely related to Jordan's
generalization of the Euler totient formula, given by 
\[
J_{d}(n)=n^{k}\prod_{\substack{\text{$p$ prime} \\ p \mid n}} \left(1 -
\frac{1}{p^{k}}\right). 
\]

\begin{lemma}\label{lem:formal.cums.cyc.ps}
The formal cumulants of the cyclotomic polynomials for $n > 1$ 
satisfy
\begin{equation}\label{eq:cyclotomic.cum.1}
\kappa_d(\Phi_{n}(q)) = \frac{B_d}{d} \sum_{k \mid n} \mu(n/k) k^d = \frac{B_d}{d} J_{d}(n).
\end{equation}
 
\begin{proof}
By \eqref{eq:mobius.inversion}, the cyclotomic polynomials for $n>1$
can be expressed in rational form as a ratio of $q$-integers so one
can use \eqref{eq:formal.cum} to compute $\kappa_d (\Phi_n(q))=
\frac{B_d}{d} \sum_{k \mid n} \mu(n/k) k^d$.  This sum factors as
\begin{equation}\label{eq:cyclotomic.cum.3} \sum_{k \mid n}
\mu(n/k) k^d = n^d \prod_{\substack{\text{$p$ prime} \\ p \mid n}}
\left(1 - \frac{1}{p^{d}}\right).
\end{equation}
Indeed, \eqref{eq:cyclotomic.cum.3} can be easily verified when $n$ is
a prime power, and it is straightforward to check that both sides are
multiplicative functions.  Thus, the second equality in \eqref{eq:cyclotomic.cum.1} holds.
\end{proof}
\end{lemma}

We will use the next corollary which was first noted by
Hwang--Zacharovas for root unitary polynomials.  It applies to all
polynomials which can be expressed as rational products of
$q$-integers.  The proof relies on \eqref{eq:formal.cum} and the fact
that the Bernoulli numbers $B_{n}$ vanish for odd $n>1$.

\begin{corollary}\cite[Cor.~3.1]{HwangZacharovas}\label{cor:kappa_alternates} Let $f \in
\bR_{\geq 0}[q]$ be any nonzero polynomial such that all of its complex roots
have modulus $1$.  The corresponding odd
cumulants $\kappa_{2d-1}(f)$ vanish after the first.  The corresponding
even cumulants $\kappa_{2d}(f)$ alternate in sign to make 
$(-1)^{d-1} \kappa_{2d}(f) \geq 0$.
\end{corollary}

\subsection{Generalized uniform sum distributions}\label{sub:uniform.sums}

Here we recall some of the notation of $p$-norms, the
decreasing sequence space with finite $p$-norm, and the
generalized uniform sum random variables.  See \cite{BS-asymptotics}
for more details.

\begin{definition}\label{def:p-norm}
  Let $\mathbf{t} = (t_1, t_2, \ldots)$ be a sequence of
  nonnegative real numbers. For $p \in \bR_{\geq 1}$, the
  \textbf{$p$-norm} of $\mathbf{t}$ is
  $|\mathbf{t}|_p \coloneqq \left(\sum_{k=1}^\infty t_k^p\right)^{1/p}$.
  We also set $|\mathbf{t}|_\infty \coloneqq \sup_k t_k$.
\end{definition}

The $p$-norm has many nice properties. 
It is well-known (e.g.~\cite[Ex.~7.3, p.58]{MR1483073})
that if $1 \leq p \leq q \leq \infty$, then
$|\mathbf{t}|_p \geq |\mathbf{t}|_q$, and that if $|\mathbf{t}|_p <
\infty$, then $\lim_{q \to \infty} |\mathbf{t}|_q =
|\mathbf{t}|_\infty$. Thus, if $\mathbf{t}$ is
weakly decreasing, $|\mathbf{t}|_\infty = \sup_k t_k = t_1$.

The \textbf{sequence space with finite $p$-norm} $\mathbf{\ell}_p
\coloneqq \{\mathbf{t} = (t_1, t_2, \ldots) \in \bR_{\geq 0}^{\bN} : |\mathbf{t}|_p <
\infty\}$ is commonly used in functional analysis and statistics.
Here we define a related concept for analyzing sums of central
continuous uniform random variables.

\begin{definition}
  The \textbf{decreasing sequence space with finite $p$-norm} is \[
\widetilde{\ell}_p \coloneqq \{\mathbf{t} = (t_1, t_2,
\ldots) : t_1 \geq t_2 \geq \cdots \geq 0, |\mathbf{t}|_p <
\infty\}. \] 
\end{definition}

The elements of $\widetilde{\ell}_p$ may equivalently be thought of as
the set of \textbf{countable multisets of nonnegative real numbers
with finite $p$-norm}. Any finite multiset of nonnegative real
numbers can be considered as an element of $\widetilde{\ell}_p$ with
finite support by sorting the multiset and appending $0$'s.  We write
$\tilde{\ell}_{\leq m}$ for the resulting collection of decreasing
sequences with at most $m$ nonzero entries. The
multisets in $\widetilde{\ell}_{p}$ are uniquely determined by their
$p$-norms.  In fact, any sequence of $p$-norm values injectively
determines the multiset provided the sequence goes to infinity.

\begin{definition}\label{def:generalized.uniform.sum}\cite[\S 3]{BS-asymptotics}
A \textbf{generalized uniform sum distribution} is any distribution
associated to a random variable with finite mean and variance given as a
countable sum of independent continuous uniform random variables.
Such random variables are given by a constant
overall shift plus a \textbf{uniform sum random variable}
\[ \cS_{\mathbf{t}} \coloneqq  \cU\left[-\frac{t_1}{2},
\frac{t_1}{2}\right] + \cU\left[-\frac{t_2}{2}, \frac{t_2}{2}\right] +
\cdots \] for some $\mathbf{t}=(t_1, t_2, \ldots) \in
\widetilde{\ell}_2$. 
\end{definition}

The uniform sum random variables have nice cumulant formulas which are
similar to the CGF distributions.   It was shown in
\cite[Lem. 3.11]{BS-asymptotics} that for  $d\geq 2$ and $\mathbf{t} =
(t_{1},t_{2},\ldots) \in \widetilde{\ell}_2$, we have
\begin{equation}\label{eq:Ucont.cumulants.2}
\kappa_d^{\cS_{\mathbf{t}}} = \frac{B_d}{d} \sum_{k=1}^\infty (t_k)^d =
\frac{B_d}{d} |\mathbf{t}|_d^d.
\end{equation}

\begin{theorem}\label{prop:multiset_infinite}\cite[Thm 3.13]{BS-asymptotics}
Generalized uniform sum distributions are bijectively parameterized by 
$\mathbb{R} \times \widetilde{\ell}_2.$ 
In particular, if
$\mathbf{t}, \mathbf{u} \in \widetilde{\ell}_2$ with $\mathbf{t} \neq
\mathbf{u}$, then $\cS_{\mathbf{t}} \neq \cS_{\mathbf{u}}$.
Furthermore, $\cS_{\mathbf{t}}^* = \cS_{\mathbf{u}}^*$ if and only if
$\mathbf{t}, \mathbf{u}$ differ by a scalar multiple.
\end{theorem}

It is not known what all the possible limiting distributions of
families of CGF polynomials are.  By \cite[Cor. 3.17]{BS-asymptotics},
we know that all standardized uniform sum distributions
$\cS_{\mathbf{t}}^*$ do occur as limiting distributions coming from
the hook length formulas for linear extensions of forests due to
Bj\"orner and Wachs \cite{MR1022316}. In fact, all standardized
\textbf{DUSTPAN distributions} can occur as limits of CGF
distributions.  These are distributions of the form
$\cS_{\mathbf{t}}+\cN(0, \sigma^{2})$, assuming the two random
variables are independent, $\mathbf{t} \in\widetilde{\ell}_2$, and
$\sigma^{2} \coloneqq 1 - |\mathbf{t}|_2^2/12 \in \mathbb{R}_{\geq
0}$.

\section{Algebraic considerations}\label{sec:algebraic}

\subsection{Equivalent characterizations of CGFs}

One algebraic justification for studying cyclotomic generating
functions as a special class of polynomials is the following classical
result of Kronecker from the 1850s. We include a proof similar to
Kronecker's for completeness. See \cite{mo:Kronecker} for further
references.

\begin{theorem} {\cite{Kronecker.1857}}\label{thm:Kronecker}
  Suppose $f(q) \in \bZ[q]$ is monic and all of its complex roots have
modulus at most $1$.  Then the roots of $f(q)$ are each either a root of
unity or $0$.
\end{theorem}

\begin{proof}
  If $f(q)=\prod_{j=1}^n (q-z_j)$ for $z_1, \ldots, z_n \in \bC$,
define $f_k(q) = \prod_{j=1}^n (q-z_j^k)$ for each positive integer
$k$. The coefficients of $f_k(q)$ are elementary symmetric polynomials
in the $z_j^k$ and each $|z_j^k| \leq 1$ by hypothesis, so the
coefficients of $f_k(q)$ are bounded in modulus by binomial
coefficients. These coefficients are also symmetric functions of
roots of $f$, so they belong to the fixed field of the Galois group of $f$,
namely the base field $\bQ$. Since $f$ is monic, its roots
$z_1, \ldots, z_n$ are algebraic integers as are all sums of products
of the $z_{i}$'s.  Therefore, the coefficients of each $f_{k}(q)$ are
also algebraic integers, and since the only algebraic integers in
$\bQ$ are the integers themselves, we observe $f_k(q) \in \bZ[q]$. The
list $f_1, f_2, \ldots$ must thus eventually repeat. We may as well
suppose $f=f_{1}$ is repeated, so that taking $k$th powers for some
$k>1$ permutes the $z_{i}$'s. Hence, taking $k$th powers $n!$ times
implies $z_{i}=z_{i}^{kn!}$ for all $i$, and the result follows.
\end{proof}

We now prove the six equivalent characterizations of a CGF from the
introduction.  These results follow closely from properties of
cyclotomic polynomials reviewed in \Cref{sec:background}.

\begin{proof}[Proof of {\Cref{thmdef:cgf_equivalences}}] By the
definition of cyclotomic polynomials, (i) $\Leftrightarrow$ (iii) and
(iv) $\Rightarrow$ (i). For (iii) $\Rightarrow$ (iv), recall the
identity for cyclotomic polynomials $\Phi_n(q) = \prod_{d \mid n} (q^d
- 1)^{\mu(n/d)}$ from \eqref{eq:mobius.inversion}.  Furthermore, if
$\Phi_{n}(q)$ is a divisor of $f \in \Phi^{+}$, then we can assume
$n>1$ since $f(1)$ is positive. By the well-known recurrence $\sum_{d
\mid n} \mu (n/d)=0$ for all $n>1$, the number of factors in the
numerator and denominator of $\Phi_n(q) = \prod_{d \mid n} (q^d -
1)^{\mu(n/d)}$ are equal so
\begin{equation}\label{eq:cyc.prod.mobius}
\Phi_n(q) = \prod_{d \mid n} (q^d - 1)^{\mu(n/d)} = \prod_{d \mid n}
[d]_q^{\mu(n/d)} = \prod_{j=1}^m \frac{1-q^{a_j}}{1-q^{b_j}} =
\prod_{j=1}^m \frac{[a_j]_q}{[b_j]_q}
\end{equation}
for some multisets $\{a_1, \ldots, a_m\}$ and $\{b_1, \ldots, b_m\}$
of positive integers. Hence (iv) is equivalent to (i) and (iii).  The
uniqueness claim follows from the uniqueness of polynomial
factorizations.

The equivalence of (i) and (iii) implies (ii) because the cyclotomic
polynomials are all monic.  In the other direction, after dividing
through by the leading coefficient which is also the greatest common
divisor of all the coefficients by hypothesis, we can assume $f(q)$ is a monic
polynomial with nonnegative integer coefficients.  If $f(q)$ also has
all of its complex roots with modulus at most $1$, then the roots of
$f$ are each either a root of unity or $0$ by
\Cref{thm:Kronecker}. Hence, (ii) implies (i).

The equivalence of (iv) and (v) follows from the polynomial identity
\[
       f(q) \prod_{j=1}^m {[b_j]_q} = \alpha q^\beta \cdot
       \prod_{j=1}^m {[a_j]_q}. 
\]
Up to a choice of $\alpha$ and $\beta$, the
equivalence of (i) and (vi) follows since the roots of unity $e^{2\pi
i k/n}$ which are zeros of $f(q)/f(1) = \bE[q^{\cX}]$ are all
determined by rational numbers $k/n$ which give rise to all zeros of
$\bE[e^{2\pi it\cX}]$.
\end{proof}

\subsection{Rational products of \texorpdfstring{$q$,q}-integers}

Consider the general class of rational products of $q$-integers of the
form
\begin{equation}\label{eq:ab_quotient}
  f(q) = \frac{\prod_{k=1}^m [a_k]_q}{\prod_{k=1}^m [b_k]_q} = \sum_{k=0}^\infty c_k q^k
\end{equation}
as formal power series in $\bZ[[q]]$. Such rational products include
the set of basic CGFs.   We examine several properties of such
products.

In the next lemma, we state an explicit formula for the coefficients
of the expansion of \eqref{eq:ab_quotient} generalizing work of Knuth
for the number of permutations with $k\leq n$ inversions in $S_n$ in
\cite[p.16]{Knuth}.  See also \cite[Ex.~1.124]{ec1} and
\cite[A008302]{oeis}, and the application to standard Young tableaux
in \cite{BKS-zeroes}.  Here we use \[ \binom{x}{k} \coloneqq
\frac{x(x-1)\cdots (x-k+1)}{k!} \] for all $k \in \bZ_{\geq 0}$ and $x
\in \bZ$, including $x<0$.  For all $x \in \mathbb{Z}$, define the
empty product $\binom{x}{0} = 1$.

\begin{lemma}\label{lem:kth_term}
  Assume $f(q) = \frac{\prod_{k=1}^m [a_k]_q}{\prod_{k=1}^m [b_k]_q} =
\sum_k c_k q^k \in \bZ[[q]]$ for multisets of positive integers $\{a_{1},\ldots,a_m
\}$ and $\{b_{1},\ldots, b_m \}$. Set \[ M_i \coloneqq \#\{k : b_k =
i\} - \#\{k : a_k = i\}. \] Then, for every $k$, the coefficient $c_{k}$ is a 
polynomial in $M_{1}, \ldots, M_k$ given by 
\begin{equation}\label{eq:kth_term} c_k = \sum_{|\mu| = k}
\prod_{\substack{i \geq 1 \\ m_i(\mu) > 0}} \binom{M_i + m_i(\mu) -
1}{m_i(\mu)} \end{equation} where $m_i(\mu)$ is the number of parts of
$\mu$ of size $i$.  Moreover, we may restrict the sum in
\eqref{eq:kth_term} to only those $|\mu| = k$ where for all $i \geq 1$, either $M_i >0$ or
$m_i(\mu) \leq |M_i|$.  
  
  \begin{proof}
By definition,    we have after cancellation 
      \[ f(q) = \prod_{M_{i}\neq 0} (1-q^i)^{-M_i}. \]
    Equation \eqref{eq:kth_term} follows using the expansion
    $(1 - q^i)^{-j} = \sum_{n=0}^\infty \binom{j+n-1}{n} q^{in}$,
    multiplication of ordinary generating functions, and the fact that
    $\binom{x}{0} = 1$. For the restriction,
    consider a term indexed by $\mu$ with $M_i \leq 0$ and
    $m_i(\mu) > |M_i| = -M_i$. It follows that
    $M_i \leq 0 \leq M_i + m_i(\mu) - 1$, so that
      \[ m_i(\mu)! \binom{M_i + m_i(\mu) - 1}{m_i(\mu)}
          = (M_i + m_i(\mu) - 1) \cdots (M_{i}+2)(M_{i}+1)  M_i = 0. \]
  \end{proof}
\end{lemma}

\begin{remark}
  The binomial coefficients $\binom{M_i + m_i(\mu) - 1}{m_i(\mu)}$
appearing in \eqref{eq:kth_term} are positive when $M_i > 0$. When
$M_i \leq 0$ and $m_i(\mu) \leq |M_i|$, the binomial coefficient is
nonzero and has sign $(-1)^{m_i(\mu)}$. Thus the negative summands
in \eqref{eq:kth_term} are precisely those for which $\mu$ has an odd
number of row lengths $i$ such that $M_i \leq 0$.
\end{remark}

\begin{remark}
  In the particular case when $f(q)$ is a basic cyclotomic generating
function, the fact that the expression in \eqref{eq:kth_term} is
nonnegative, symmetric, and eventually $0$ is quite remarkable.  Is
there a simplification of the expression in \eqref{eq:kth_term} that
would directly imply nonnegativity of the coefficients for any $f(q)
\in \Phi^{+}$?  Even the case of $[n]_{q}!$ or any of the examples
CGFs mentioned in the Introduction would be of interest.
\end{remark}

\subsection{Necessary conditions for \texorpdfstring{$q$}-integer ratios to be polynomial}\label{ssec:rational_props}
The most obvious necessary and sufficient condition for rational
products of $q$-integers of the form
\begin{equation}\label{eq:q_quotient}
f(q)=\frac{\prod_{k=1}^m [a_k]_q}{\prod_{k=1}^m
[b_k]_q} 
\end{equation}
to yield a polynomial (not necessarily with positive coefficients) is
for the multiplicity of each primitive $d$th root of unity to be
nonnegative. However, both numerator and denominator can be factored
uniquely into cyclotomic polynomials, so the \textbf{polynomiality criterion} is
equivalent to
\begin{equation}\label{eq:multiplicity_dominance}
  \#\{k : \ell \mid a_k\} \geq \#\{k : \ell \mid b_k\} \hspace{.1in}
  \qquad\forall \ \ell \in \bZ_{\geq 2}.
\end{equation}

\begin{example}\label{ex:hooks}
Recall from the introduction that Stanley's $q$-analogue of the hook
length formula
$$
q^{b(\lambda)} \frac{[n]_q!}{\prod_{c \in \lambda} [h_c]_q} = \sum_{T
\in \mathrm{SYT}(\lambda )} q^{\maj (T)}, $$ gives an important family
of cyclotomic generating functions.  In this case, the inequality
\eqref{eq:multiplicity_dominance} reduces to the fact that the
$\ell$-quotient of any partition $\lambda$ with $|\lambda| = n$ has no
more than $\lfloor n/\ell\rfloor$ cells. In particular, this gives a
quick proof that $[n]_q!/\prod_{c \in \lambda} [h_c]_q \in \bZ[q]$,
though the stronger result that the coefficients are nonnegative is
not clear from this approach.  See \cite[Prop.~4.11]{stanley.1979} and
\cite[Cor. 21.5]{ec2} for more details and proofs for this equality.
 \cite{BKS-asymptotics} for

\end{example}

\subsection{Necessary conditions for a \texorpdfstring{$q$}-integer ratio to be a CGF}\label{ssec:cgf_props}
In addition to satisfying the polynomiality conditions in
\eqref{eq:multiplicity_dominance} and \eqref{eq:ineq_d_powers}, what
further restrictions on $\{a_1, \ldots, a_m\}$ and $\{b_1, \ldots,
b_m\}$ are required for a rational product of $q$-integers as in
\eqref{eq:q_quotient} to yield a cyclotomic generating function?  Say
$f(q)=\prod_{k=1}^m [a_k]_q/\prod_{k=1}^m [b_k]_q = \sum c_{k}
q^{k}$.  Using the expression for $c_k$ in \eqref{eq:kth_term}, $c_k
\geq 0$ and $\lim_{k \to \infty} c_k = 0$ are by definition necessary and
sufficient to prove $f(q) \in \Phi^{+}$, though these conditions are
difficult to use for families of such rational products in practice.
We will outline some more direct necessary conditions here.
We begin with complete classifications for $m=1, 2$.

\begin{lemma}\label{lem:m=1}
  A rational product of $q$-integers $[a]_q/[b]_q$ is a CGF if and only if $b \mid a$.
\end{lemma}

\begin{proof}
  If the ratio is a CGF, it must evaluate to a positive integer at $q=1$, so $b \mid a$
  is necessary. It is also sufficient since when $b \mid a$, we have $[a]_q/[b]_q = [a/b]_{q^b}$, which is a CGF.
\end{proof}

\begin{lemma}
  Consider a rational product of $q$-integers $f(q) = [a_1]_q [a_2]_q / [b_1]_q [b_2]_q$.
  \begin{enumerate}
  \item Then, $f(q)$ is a polynomial if and only if
  \begin{enumerate}[(i)]
    \item $b_1 \mid a_1$ and $b_2 \mid a_2$; or
    \item $b_1 \mid a_2$ and $b_2 \mid a_1$; or
    \item $b_1, b_2 \mid a_1$ and $\gcd(b_1, b_2) \mid a_2$; or
    \item $b_1, b_2 \mid a_2$ and $\gcd(b_1, b_2) \mid a_1$.
  \end{enumerate}
  \item If $f(q)$ is a power series with nonnegative coefficients,
  then $a_1, a_2 \in \Span_{\bZ_{\geq 0}}\{b_1, b_2\}$.
  \item Moreover, $f(q) \in \Phi^{+}$ if and only if both the divisibility and span conditions hold.
  \end{enumerate}

\end{lemma}

\begin{proof}
  First suppose $f(q)$ is a polynomial. If either $b_{i}=1$, then
$[1]_{q}=1$ and the conditions follow easily by \eqref{eq:q-integers}, so
assume each $b_{i}>1.$ By \eqref{eq:q-integers}, $\Phi_{b_i}(q)$
divides $[b_{i}]$. Since $\Phi_{b_i}(q)$ appears in the denominator of
$f(q)$, it must be canceled in the numerator, forcing $b_i \mid a_1$
or $a_2$ for $i=1, 2$. If we do not have cases (i) or (ii), we have
the first clause of (iii) or (iv). Without loss of generality suppose
$b_1, b_2 \mid a_1$ and $\gcd(b_1, b_2) > 1$. Then, $\Phi_d(q)$ for
$d=\gcd(b_1, b_2)$ appears twice in the denominator, and so must
divide both $[a_1]_q$ and $[a_2]_q$ in order for $f(q)$ to be
polynomial, giving $d \mid a_2$ by \eqref{eq:q-integers} again.
  
 Conversely, suppose one of (i)-(iv) holds. Then, $f(q)$ is a
polynomial if conditions (i) or (ii) hold by \Cref{lem:m=1}, and (iv)
is equivalent to (iii), so it suffices to assume (iii) holds. Use the
cyclotomic expansion of $[n]_{q}$ again from \eqref{eq:q-integers}.
Since (iii) holds, every cyclotomic divisor of either $[b_1]_{q}$ or
$[b_2]_{q}$ but not both is canceled by a cyclotomic divisor of
$[a_1]_{q}$, and every divisor of both $[b_1]_{q}$ and $[b_2]_{q}$ is
canceled by divisors of $[a_1]_{q}$ and $[a_2]_{q}$ together.  Thus,
(1) holds.

The span condition in statement (2) is proved for all basic CGFs in
\Cref{lem:cgf_ineqs} below. To prove (3), observe that $f(q) \in
\Phi^{+}$ implies both the divisibility and span conditions hold by
(1) and (2).  Conversely, if any one of the divisibility
conditions hold, then $f(q)$ is a polynomial with integer
coefficients. If divisibility conditions (i) or (ii) hold, then $f(q)
\in \Phi^{+}$ by \Cref{lem:m=1}, so suppose divisibility condition
(iii) and the span condition in (2) holds.  Again the case (iv) is
similar.

By canceling out common cyclotomic factors, we may suppose without
loss of generality that $\gcd(b_1, b_2) = 1$ and that $a_1 = b_1
b_2$. Since $a_{2} \in \Span_{\bZ_{\geq 0}}\{b_1, b_2\}$, there exist
positive integers $u,v$ such that $a_{2}= u b_1 + v b_2$.  So, we must
show that the polynomial $f(q)=[b_1 b_2]_q [u b_1 + v b_2]_q / [b_1]_q
[b_2]_q$ has nonnegative coefficients.  

Given the assumptions above, $f(q)$ may be expressed as
\[
f(q)=(1-q^{b_1 b_2})(1-q^{u
b_1 + v
b_2})(1+q^{b_1}+q^{2b_1}+\cdots)(1+q^{b_2}+q^{2b_2}+\cdots).
\]
Let $N(n) = \#\{x, y \in \bZ_{\geq 0} : xb_1 + yb_2 = n\}$. The
coefficient of $q^n$ in $f(q)$ is hence
\[
N(n) - N(n - b_1 b_2) - N(n - ub_1 - vb_2) + N(n - b_1 b_2 - ub_1 - vb_2).
\]
Let $N(n; \alpha, \beta) = \#\{\alpha \leq x, \beta \leq y : xb_1 +
yb_2 = n\}$. The previous expression becomes \[ N(n; 0, 0) - N(n; b_2,
0) - N(n; u, v) + N(n; u+b_2, v). \] Furthermore, let $N(n; \alpha,
\beta; \delta) = \#\{\alpha \leq x < \alpha+\delta, \beta \leq y :
xb_1 + yb_2 = n\}$. Thus, the expression for the coefficient of $q^n$
in $f(q)$ becomes \[ N(n; 0, 0; b_2) - N(n; u, v; b_2). \]

Assume there exist $(x_0, y_0) \in \bZ^{2}$ such that $x_{0}b_1 +
y_{0}b_2 = n$, since otherwise the coefficient of $q^n$ in $f(q)$ is $0$.
Then, all integer solutions of $xb_1 + yb_2 = n$ are of the form
$x=x_0-tb_2, y=y_0+tb_1$ for some $t \in \bZ$ by the theory of linear
Diophantine equations.  In particular, there exist unique solutions
$(x_1, y_1), (x_2, y_2) \in \bZ^{2} $ with $0 \leq x_1 < b_2$ and $u
\leq x_2 < u+b_2$. Hence, $N(n; 0, 0; b_2) = \delta_{y_1 \geq 0}$ and
$N(n; u, v; b_2) = \delta_{y_2 \geq v}$. Since $x_1 \leq x_2$, we have
$y_1 \geq y_2$ as solutions to the linear equation. Thus, $\delta_{y_2
\geq v} = 1 \Rightarrow \delta_{y_1 \geq 0} = 1$.  Therefore, $ N(n;
0, 0; b_2) - N(n; u, v; b_2)$ is nonnegative, which completes the
proof.
\end{proof}

\begin{example}
For $m=2$, the polynomiality condition alone does not imply
nonnegative integer coefficients.  For example,
$[1]_{q}[6]_{q}/[2]_{q}[3]_{q}=\Phi_{6}(q)=q^{2}-q+1$.  Here $6 \in
\Span_{\bZ_{\geq 0}}\{2,3\}$, but $1$ is not.  For $m=3$ and higher,
the polynomiality and span conditions do not imply nonnegative
integer coefficients. For example,
\begin{equation}\label{eq:m3}
\frac{[4]_q[4]_q[15]_q}{[2]_q[3]_q[5]_q} =
q^{13} + q^{11} + q^{10}  -q^{9} + 2q^{8} + 2q^{5}  -q^{4} + q^{3} + q^{2} + 1.
\end{equation}
However, we do have the following necessary conditions for basic CGFs
related to spans and sums.  
\end{example}

In the next lemma, we will use the big Theta notation for asymptoticly
bounding a function.  If $f(n)$ is a function of $n$, we write
$f=\Theta(g)$ provided there exist constants $c,d$, a function $g(n)$,
and a positive integer $n_{0}$ such that for all $n\geq n_{0}$ we have
$cg(n)\leq f(n) \leq dg(n)$.  Thus, the function $g$ is a reasonable
approximation to $f$ for large $n$.  We will extend the big Theta
notation to other families of functions not necessarily indexed by
positive integers if such constants exist.

\begin{lemma}\label{lem:cgf_ineqs}
  If  $f(q)=\prod_{k=1}^m [a_k]_q/[b_k]_q \in \Phi^{+}$,  then the following hold.
  \begin{enumerate}[(i)]
    \item For each $1\leq  k\leq m$,\ $a_k \in \Span_{\bZ_{\geq 0}}\{b_1, \ldots, b_m\}$.
     \item\label{part:ends} If $a_1 \leq \cdots \leq a_m$ and $b_1 \leq \cdots \leq b_m$,
           then $a_1 \geq  b_1\geq 1$ and $a_m \geq  b_m$.  

    \item  Let $\cX$ be the random variable corresponding to $f(q)/f(1)$
    with mean $\mu =\frac{1}{2}\sum_{k=1}^m(a_k - b_k)$ and variance $\sigma^{2}=\frac{1}{6}\sum_{k=1}^m(a_k^{2} - b_k^{2})$. Then, $\log\sigma = \Theta(\log\mu)$. More precisely,
      \[ \mu/2 \leq \sigma^2 \leq \mu^2. \]

    \item We have \[ \frac{\sum_{i=1}^m
(a_i^4 - b_i^4)} {\left(\sum_{i=1}^m (a_i^2 - b_i^2)\right)^2} \leq
\frac{5}{3}. \]

\item For all positive even integers $d$ and $d=1$, we have 
\begin{equation}\label{eq:ineq_d_powers}
\sum_{k=1}^m a_k^d \geq
\sum_{k=1}^m b_k^d,
\end{equation}
where the inequality is strict if the polynomial is non-constant.  For
example, when $d=1$ this inequality is equivalent to noting the degree
of $f(q)$ is $\sum_{k=1}^m a_k - b_k \geq 0$.
\end{enumerate}
  
  \begin{proof}
    (i) From the definition of $f(q)$, we have 
      \[ \prod_{k=1}^m 1/(1-q^{b_k}) = f(q) \prod_{k=1}^m 1/(1-q^{a_k}). \]
    The support of the power series $\prod_{k=1}^m 1/(1-q^{b_k})$
    is $\Span_{\bZ_{\geq 0}} \{b_1, \ldots, b_m\}$. Since $f(q)$
    has nonnegative coefficients and constant term $1$, comparing supports gives
      \[ \Span_{\bZ_{\geq 0}} \{a_1, \ldots, a_m\} \subset \Span_{\bZ_{\geq 0}} \{b_1, \ldots, b_m\}, \]
    which yields (i).
    
    (ii) From (i), it follows that $a_1 \geq \min\{b_1, \ldots, b_m\} = b_1$.
    From \eqref{eq:multiplicity_dominance} at $\ell=b_m$, there is
    some $a_k$ such that $b_m \mid a_k \leq a_m$, so $b_m \leq a_m$.

    (iii) The argument in \cite[Lem.~2.5]{HwangZacharovas} applies here.

    (iv) By definition, the \textit{kurtosis} of $\cX$ is
$\bE[(\frac{X-\mu}{\sigma})^{4}]$. Jensen's Inequality for convex
functions gives the inequality of central moments
\[
\alpha_2^2 = \bE[(\cX - \mu)^2]^2 \leq \bE[(\cX - \mu)^4] = \alpha_4.
\]
In terms of cumulants, this says $\kappa_2^2 \leq \kappa_4 + 3
\kappa_2^2$, which simplifies to the stated expression by
\eqref{eq:prod_cumulants}.

(v) The inequality follows directly from the formula for cumulants of
a rational product of $q$-integers \eqref{eq:prod_cumulants} and the
sign constraints in \Cref{cor:kappa_alternates}.
  \end{proof}
\end{lemma}

\begin{remark}
  Warnaar--Zudilin \cite[Conj.~1]{MR2773098} conjecture that
  a family of inequalities similar to \eqref{eq:multiplicity_dominance}
  for ratios of $q$-factorials implies positivity. They prove the conjecture
  for $q$-super Catalan numbers.
\end{remark}

\begin{remark}\label{rem:cgf_not_pointwise}
One might wonder if $f(q)=\prod_{k=1}^m [a_k]_q/[b_k]_q \in \Phi^{+}$
implies $a_k \geq b_k$ for all $k$ when sorted as in
\Cref{lem:cgf_ineqs}(ii). However, this is \textbf{false}, though
empirically counterexamples are rare. For instance, \[
\frac{[2]_q[3]_q[3]_q[8]_q[12]_q}{[1]_q[1]_q[4]_q[4]_q[6]_q} = q^{12}
+ 2 q^{11} + 2 q^{10} + 2 q^{7} + 4 q^{6} + 2 q^{5} + 2 q^{2} + 2 q +
1 \] is one of only three counterexamples with $5$ $q$-integers in the
numerator and denominator and maximum entry $\leq 14$, out of $956719$
cyclotomic generating functions satisfying these conditions. On the other
hand, we propose the following.
\end{remark}

\begin{conjecture}\label{conj:majorization}
Suppose $f(q)=\prod_{k=1}^m [a_k]_q/[b_k]_q \in \Phi^{+}$,\ $a_1
\leq \cdots \leq a_m$, \ $b_1 \leq \cdots \leq b_m$.  Then
$\sum_{k=1}^\ell a_k \geq \sum_{k=1}^\ell b_k$ and $\sum_{k=\ell}^m
a_k \geq \sum_{k=\ell}^m b_k$ for all $\ell$. That is, $\{a_1,
\ldots, a_m\}$ (weakly) \textbf{majorizes} $\{b_1, \ldots, b_m\}$ from
both sides.
\end{conjecture}

  The majorization inequalities hold for $\ell=1$ by
  \Cref{lem:cgf_ineqs}(ii) and $\ell=m$ by
\eqref{eq:ineq_d_powers}.  These majorization 
inequalities have also been checked for multisets with $7$ elements
and largest entry at most $15$. Out of the $\binom{15+7-1}{7}^2 =
13521038400$ pairs of multisets, $70653669$ or roughly $0.5\%$ yield
cyclotomic generating functions.  Of these, $2713$ do not satisfy $a_k
\geq b_k$ for all $k$, or roughly $0.004\%$ of the cyclotomic
generating functions. Each of these nonetheless satisfy the
majorization inequalities in \Cref{conj:majorization}. Moreover,
\Cref{conj:majorization} holds for all $f$ of degree $\leq 42$,
of which there are $10439036$. Finally, the majorization condition
is preserved under multiplication.

\begin{remark}
  The affirmative answer to \Cref{conj:majorization} together with
  Karamata's inequality would give
  $\sum_{k=1}^\ell \psi(a_k) \geq \sum_{k=1}^\ell \psi(b_k)$
  and $\sum_{k=\ell}^m \psi(a_k) \geq \sum_{k=\ell}^m \psi(b_k)$
  for every non-decreasing convex function $\psi \colon [1, \infty) \to \bR$,
  in particular strengthening \eqref{eq:ineq_d_powers} and leading to
  more necessary conditions for CGFs.
\end{remark}

\subsection{CGFs with a fixed denominator}\label{ssec:CGF_denom}

In a different direction, one may fix the denominator
$\prod_{k=1}^m [b_k]_q$ and consider which
numerators $\prod_{k=1}^m [a_k]_q$
yield cyclotomic generating functions. Fixing the multiset
$B=\{b_1, \ldots, b_m\}$, let $G_{B}$ be the graph whose
vertices are multisets $\{a_1, \ldots, a_m\}$ for which
$\prod_{k=1}^m [a_k]_q/[b_k]_q \in \Phi^{+}$
and two vertices are connected by an edge if their
multisets differ by a single element.

\begin{lemma}\label{lem:cgf_graph}
  Fix a multiset $B=\{b_1, \ldots, b_m\}$ of positive integers.
  Suppose $\{a_1, \ldots, a_m\}$ and $\{a_1', \ldots, a_m'\}$
  are multisets of positive integers such that both 
  $\prod_{k=1}^m [a_k]_q/[b_k]_q$ and   $\prod_{k=1}^m [a_k']_q/[b_k]_q
  \in \Phi^{+}$ are cyclotomic generating functions.
  Then there exists a multiset $\{h_1, \ldots, h_m\}$ of positive
  integers such that for all $1 \leq i \leq m$,
    \[ \frac{\prod_{k=1}^i [h_k]_q \prod_{k=i+1}^m [a_k]_q}
               {\prod_{k=1}^m [b_k]_q} \in \Phi^{+} \qquad \text{and} \qquad
        \frac{\prod_{k=1}^i [h_k]_q \prod_{k=i+1}^m [a_k']_q}
               {\prod_{k=1}^m [b_k]_q} \in \Phi^{+}. \]
  
  \begin{proof} By long division, one may observe $[xy]_q/[y]_q =
[x]_{q^y} \in \Phi^{+}$.  So, we may replace any $[a_{k}]_q$ in the
numerator of a cyclotomic generating function with $[\ell a_{k}]_q$
for any positive integer $\ell$ and get another cyclotomic generating
function. Hence, defining $h_k = a_k a_k'$ for $1\leq k\leq m$ has the
desired property.  \end{proof}
\end{lemma}

\begin{corollary}
The graph $G_{B}$ is nonempty, connected and has diameter at most $2m$.
\end{corollary}

\section{Asymptotic considerations}\label{sec:asymcgfs}

Given a sequence of cyclotomic generating functions and their
corresponding random variables, we can put their distributions in one
common frame of reference by standardizing them to have mean 0 and
standard deviation 1.  Under what conditions does such a sequence
converge?  Can we classify all possible standardized limiting
distributions of random variables coming from CGFs?  Both problems
have been completely solved for certain families of CGFs coming from
$q$-hook length formulas, but the complete classification is not known,
see \cite[Open Problem 1.19]{BS-asymptotics}.  In several ``generic''
regimes, they are asymptotically normal
\cite[Thm.~1.7]{BKS-asymptotics}, \cite[Thm.~1.13]{BS-asymptotics}. In
mildly degenerate regimes, they are related to independent sums of
uniform distributions \cite[Thm.~1.8]{BS-asymptotics}.  Our aim in
this section is to give another ``generic'' asymptotic normality
criterion for sequences of CGFs, which is sufficient to quickly
identify the limit in many cases of interest.

Throughout the rest of the section, we will use the following
notation.  Let $a^{(N)}$ and $b^{(N)}$ for $N=1,2,\dots$ denote two
sequences of multisets of positive integers of the same size.
By \Cref{lem:cgf_ineqs}(ii), we can always assume the values in
$a^{(N)}$ are strictly greater than 1.  The
multisets $b^{(N)}$ may contain 1's in order to have the same size as
$a^{(N)}$.  Let $f_{N}(q)= \prod_{a \in a^{(N)}} [a]_q/\prod_{b \in
b^{(N)}} [b]_q$, and let $\cX_N$ be the random variable associated to
$\bE[q^{\cX_N}] =f_{N}(q)/f_{N}(1)$.  We will denote the $d$th moment
and cumulant of $\cX_{N}$ by $\mu_d^{(N)}$ and $\kappa_{d}^{(N)}$ for
$N=1,2,\ldots$ respectively.  Similarly, the standard deviation and
mean of $\cX_{N}$ are denoted by $\sigma^{(N)}$ and $\mu^{(N)}$.
Recall the notation for uniform sum distributions $\cS_{\mathbf{t}}$
and $p$-norms from \Cref{sub:uniform.sums}.

\begin{theorem}\label{thm:q_prod_an}
 Let $a^{(N)}$ be a sequence of multisets of positive integers where
$a \geq 2$ for each $a \in a^{(N)}$.  Let $f_N(q)= \prod_{a \in
a^{(N)}} [a]_q$ and $\cX_{N}$ be the associated sequence of CGF
polynomials and random variables.

\begin{enumerate}[(i)]
    \item The sequence of random variables $\cX_1, \cX_2, \ldots$ is
      asymptotically normal if
      \begin{equation}\label{eq:q_prod_an}
        \sum_{a \in a^{(N)}} \left(\frac{a}{\max a^{(N)}}\right)^4 \to \infty.
      \end{equation}
    \item Suppose the cardinality of the multisets $a^{(N)}$ is
    bounded by some $m$ and
      $\max a^{(N)} \to \infty$. Then $\cX_1^*, \cX_2^*, \ldots$
      converges in distribution if and only if the rescaled multisets
      $a^{(N)}/|a^{(N)}|_2 \in \widetilde{\ell}_{\leq m}$
      converge pointwise to some multiset $\mathbf{t} \in \widetilde{\ell}_{\leq m}$
      in the sense of \cite[\S3.2]{BS-asymptotics}. In that case,
      the limiting distribution is the standardized uniform sum $\cS_{\mathbf{t}}^*$.  
    \item Suppose the cardinality of $a^{(N)}$ is bounded and $\max a^{(N)}$
      is also bounded. Then $\cX_1^*, \cX_2^*, \ldots$ converges in distribution if
      and only if $a^{(N)}$ is eventually a constant multiset.
  \end{enumerate}

  \begin{proof} The polynomial $f_N(q) = \prod_{a \in a^{(N)}} [a]_q$
can also be expressed in rational form as $=\prod_{a \in a^{(N)}}
[a]_q/[1]_{q}$. Thus, we obtain the corresponding sequence of
multisets $b^{(N)}=(1,1,\dots,1)$ from the denominators where the size
of $b^{(N)}$ equals the size of $a^{(N)}$ as multisets.  By
\Cref{thm:prod_cumulants} and the scaling property of cumulants, the
standardized cumulants of $\cX_N$ are given by
\begin{equation}\label{eq:seq.norm.cum}
      (\kappa_d^{(N)})^*
        = \frac{\kappa_d^{(N)}}{(\sigma^{(N)})^{d/2}}
        = \frac{\frac{B_d}{d} \sum_{a \in a^{(N)}} (a^d - 1)}
           {\left(\frac{B_2}{2} \sum_{a \in a^{(N)}} (a^2 -
	   1)\right)^{d/2}}.
\end{equation}
Since $a \geq 2$ for all $a \in a^{(N)}$, $\frac{1}{2} a^d \leq a^d -
1 \leq a^d$, so $a^d - 1 = \Theta(a^d)$ uniformly. Hence for $d \geq
2$ even,
\[
(\kappa_d^{(N)})^* = \Theta\left(\frac{\sum_{a \in
a^{(N)}} a^d}{(\sum_{a \in a^{(N)}} a^2)^{d/2}}\right) =
\Theta\left(\frac{|a^{(N)}|_d^d}{|a^{(N)}|_2^d}\right) =
\Theta\left(\frac{|a^{(N)}|_d}{|a^{(N)}|_2}\right)^d.
\]

Therefore by \Cref{thm:4thcumulant.test}, asymptotic normality is
equivalent to $|a^{(N)}|_4/|a^{(N)}|_2 \to 0$, or equivalently
\[
\frac{|a^{(N)}|_2}{|a^{(N)}|_4} \to \infty.
\]
By \cite[(3.7)]{BS-asymptotics},
\[
|a^{(N)}|_4 \leq |a^{(N)}|_\infty^{1/2}\ |a^{(N)}|_2^{1/2}.
\]
Hence, squaring both sides and rearranging the factors we have 
\[
\frac{|a^{(N)}|_4}{|a^{(N)}|_\infty} \leq
\frac{|a^{(N)}|_2}{|a^{(N)}|_4}.
\]
Thus, $|a^{(N)}|_4/|a^{(N)}|_\infty \to \infty$ implies asymptotic
normality.  By Definition~\ref{def:p-norm}
\[
\left(|a^{(N)}|_4/|a^{(N)}|_\infty \right)^{4} =   \sum_{a \in a^{(N)}} \left(\frac{a}{\max a^{(N)}}\right)^4
\]
so (i) holds.

For (ii), let $\hat{a}^{(N)} \coloneqq \sqrt{12} \cdot
a^{(N)}/|a^{(N)}|_2$ be the rescaled multiset such that
$\cS_{\hat{a}^{(N)}}$ is standardized for each $N$ according to \eqref{eq:Ucont.cumulants.2}. Since $|a^{(N)}|_2
\geq |a^{(N)}|_\infty := \max a^{(N)} \to \infty$, we have by \eqref{eq:seq.norm.cum} that
\begin{align*}
(\kappa_d^{(N)})^* &=
\frac{\frac{B_d}{d} \sum_{a \in a^{(N)}} (a^d-1)}
           {\left(\frac{B_2}{2} \sum_{a \in a^{(N)}} (a^2 -1)
	   \right)^{d/2}} \\
&\sim \frac{\frac{B_d}{d} \sum_{a \in a^{(N)}} (a/|a^{(N)}|_2)^d}
{\left(\frac{B_2}{2} \sum_{a \in a^{(N)}}
(a/|a^{(N)}|_2)^2\right)^{d/2}} =
\frac{\frac{B_d}{d}}{\left(\frac{B_2}{2}\right)^{d/2}}
\frac{|\hat{a}^{(N)}|_d^d}{|\hat{a}^{(N)}|_2^{d/2}}.
\end{align*}
Here and elsewhere we write $f(N) \sim g(N)$ to mean $\lim_{N \to \infty} f(N)/g(N) = 1$.
This last expression is the $d$th cumulant of $\cS_{\hat{a}^{(N)}}$ by
\eqref{eq:Ucont.cumulants.2}.  Therefore, by the cumulant version of
\Cref{thm:moments} and \Cref{thm:frechet_shohat_converse}, we know
$\cX_1^*, \cX_2^*, \ldots$ converges in distribution if and only if
the sequence $\cS_{\hat{a}^{(N)}}$ converges in distribution, and to
the same limit. By \cite[Thm.~1.15]{BS-asymptotics},
$\cS_{\hat{a}^{(N)}}$ converges in distribution if and only if the
multisets $\hat{a}^{(N)}$ converge pointwise in the sense of
\cite[\S3]{BS-asymptotics} to some $\mathbf{t} \in \tilde{\ell}_2$, with
limiting distribution $\cS_{\mathbf{t}} + \cN(0, \sigma^2)$ where $\sigma =
\sqrt{1 - |\mathbf{t}|_2^2/12}$. Since the cardinality of $a^{(N)}$ is bounded,
$|\hat{a}^{(N)}|_2^2 = 12$ for each $N$, so $|\mathbf{t}|_2^2 = \lim_{N \to
\infty} |a^{(N)}|_2^2 = 12$.  This implies $\sigma =0$ by definition,
so the limiting distribution of the standardized random variables
$\cX_N^*$ is just the standardized generalized uniform sum
$\cS_{\mathbf{t}}=\cS_{\mathbf{t}}^*$ as given in the statement.

    For (iii), if both the cardinality of $a^{(N)}$ and
$|a^{(N)}|_\infty$ are bounded, there are only finitely many possible
choices for $a^{(N)}$.  Hence, convergence in distribution happens if
and only if the sequence $a^{(N)}$ is eventually constant.
\end{proof}
\end{theorem}

Building on the case of finite products of $q$-integers above, we can now
address the characterization of asymptotic normality for basic CGFs
stated in the introduction. Intuitively, the idea is that if the
numerator of the rational form of the CGF leads to asymptotic
normality, then the rational form does as well since it is dominated
by its numerator.

\begin{proof}[Proof of \Cref{thm:diaconis}] Observe that for any positive
integer $n$, the polynomial $[n]_q/n$ is the probability generating
function of a discrete uniform random variable $\cU_n$ supported on
$\{0, 1, \ldots, n-1\}$.  Let $\cA_{N} = \sum_{a \in a^{(N)}} \cU_a$\
and $\cB_{N} = \sum_{b \in b^{(N)}} \cU_b$ denote the CGF
distributions corresponding to the numerator and denominator of
$f_N(q)= \prod_{a \in a^{(N)}} [a]_q/\prod_{b \in b^{(N)}} [b]_q$.  By
\Cref{thm:prod_cumulants}, $\sigma_{\cA_{N}}^2 = \frac{1}{12}\sum_{a \in a^{(N)}}
(a^2 - 1)$ and $\sigma_{\cB_{N}}^2 = \frac{1}{12}\sum_{b \in b^{(N)}} (b^2 -
1)$.

Let $c_N \coloneqq \sigma_{\cB_{N}}/\sigma_{\cA_{N}}$, so each
$c_{N}\geq 0$.  By the hypothesis in \eqref{eq:diaconis.1},
$$
\limsup_{N \to \infty} c_N^{2} = \limsup_{N \to \infty} 
\frac{\sum_{b \in b^{(N)}} (b^2 - 1)} {\sum_{a \in a^{(N)}} (a^2 - 1)}
< 1.
$$
Recall that convergence in distribution
of real-valued random variables can be metrized using the L\'evy
metric. Therefore, it suffices to show that every subsequence of
$\cX_1, \cX_2, \ldots$ itself has an asymptotically normal
subsequence.  Hence, without loss of generality, we may assume that
the sequence $c_N$ converges to some $0 \leq c < 1$.

By construction, we have the following equality in distribution:
\begin{equation}\label{eq:di.sums}
\cX_N + \cB_{N} = \cA_{N},
\end{equation}
where all summands are independent.  It follows from independence that
\[ \sigma_{\cX_N}^2 =
\sigma_{\cA_{N}}^2 - \sigma_{\cB_{N}}^2.
\]
By standardizing random variables and simplifying \eqref{eq:di.sums}, we
have
\begin{equation}\label{eq:diaconis.2} \cX_N^*\ \sqrt{1-c_N^2} + \cB_{N}^*\ c_N = \cA_{N}^*.
\end{equation}
In terms of characteristic functions, \eqref{eq:diaconis.2}
gives \begin{equation}\label{eq:diaconis.3} \phi_{\cX_N^*}(t) =
\frac{\phi_{\cA_{N}^*}\left(t\, /\sqrt{1-c_N^2} \right)}
{\phi_{\cB_{N}^*}\left(t \, c_N/\sqrt{1-c_N^2} \right)}.
\end{equation}

Since the hypothesis in \eqref{eq:diaconis.1b} implies
\eqref{eq:q_prod_an} holds, \Cref{thm:q_prod_an} says the sequence
$\cA_{N}$ is asymptotically normal, so by L\'evy's continuity
\Cref{thm:levy}, \[ \lim_{N \to \infty} \phi_{\cA_{N}^*}(t) = \exp(-t^2/2) \]
for all $t \in \bR$. This convergence is in fact uniform on bounded
subsets of $\bR$ (see e.g.~\cite[Exercise~26.15(b)]{MR1324786}).
Consider the cases for $0\leq c<1$.

    \begin{enumerate}[(i)]
      \item Suppose $c>0$. By \Cref{lem:cgf_ineqs}(\ref{part:ends}), we have
        $|b^{(N)}|_{\infty } \leq |a^{(N)}|_{\infty}$. This observation and the fact that $c>0$
        imply that \eqref{eq:diaconis.1b} holds with $a^{(N)}$ replaced by $b^{(N)}$.
        Hence
          \[ \lim_{N \to \infty} \phi_{\cB_{N}^*}(t) = \exp(-t^2/2). \]
        Since $c<1$ and convergence is uniform on bounded subsets,
        \eqref{eq:diaconis.3} gives
          \[ \lim_{N \to \infty} \phi_{\cX_N^*}(t)
              = \frac{\exp(-t^2/ 2(1-c^2))}{\exp(-t^2 c^2/ 2(1-c^2))}
              = \exp(-t^2/2). \]
        Therefore, by L\'evy's continuity theorem again, $\cX_1,
	\cX_2, \ldots$ is asymptotically normal.

\bigskip

      \item Suppose $c=0$. Note that the variance of $\cB_{N}^* c_N$
is $c_N^{2}$ and the mean is 0. Hence, $c_N \to c = 0$ implies
$\cB_{N}^* c_N$ converges to the constant random variable $0$ by
Chebyshev's Inequality. Thus,
  \[ \lim_{N \to \infty} \phi_{\cB_{N}^* c_N}(t) = 1, \]
and the result follows from \eqref{eq:diaconis.3} and the
calculations from the previous case.\qedhere
    \end{enumerate}
  \end{proof}

\begin{example}
Consider a sequence of $q$-binomial coefficients
\[
f_{N}(q) = \binom{n}{k}_q=\prod_{j=1}^{k}\frac{[n-k+j]_{q}}{[j]_{q}}
\]
where $k,n$ represent sequences indexed by $N$.  Assume both $k$ and
$n-k$ approach infinity as $N \to \infty$.  In the notation of the proof
above, we have $\cA_{N} = \sum_{j=1}^{k} \cU_{n-k+j}$,\ $\cB_{N} =
\sum_{j=1}^{k} \cU_j$, and
\begin{align*}
c_{N}^{2}=\frac{\sigma_{\cB_{N}}^2}{\sigma_{\cA_{N}}^2} =
\frac{\sum_{j=1}^k j^2 - 1}{\sum_{j=n-k+1}^n j^2 - 1} \sim
\frac{k^3}{n^3 - (n-k)^3} = \frac{(k/n)^3}{1 - (1 -
k/n)^3}.
\end{align*}
Let $x = k/n$.  Since both $k, n-k \to \infty$ as $N \to \infty$, we
may suppose $k\leq n-k$ for all $N$ since we can replace $k$ with
$n-k$ if necessary without changing $f_{N}(q)$, so $0\leq x \leq 1/2$.
It is easy to check that $x^3/(1 - (1-x)^3) \leq 1/7$ for $0 \leq x
\leq 1/2$, so that \eqref{eq:diaconis.1} holds. We also see

\begin{align*}
    \frac{\sum_{j=n-k+1}^n j^4}{n^4} &\sim \frac{n^5 - (n-k)^5}{n^4}\\
        &= \frac{k (5n^4 - 10n^3 k + 10n^2 k^2 - 5n k^3 + k^4)}{n^4} \\
        &= k (5 - 10x + 10x^2 - 5x^3 + x^4),
\end{align*}
  so \eqref{eq:diaconis.1b} holds as $k \to \infty$ since $(5 - 10x +
10x^2 - 5x^3 + x^4)$ is positive for all real values in the range
$0\leq x\leq 1/2$. Thus, \Cref{thm:diaconis} applies to show the
sequence of CGF random variables corresponding to $f_{N}(q)$ is
asymptotically normal provided $k, n-k \to \infty$. The details of
this calculation are left implicit in \cite{MR964069}.
\end{example}

\section{Analytic considerations}\label{sec:analytic}

In this section, we consider the problem of classifying all
standardized CGF distribution via analytic considerations of their
characteristic functions $\phi_{\cX^*}(t)$ and their second
characteristic functions, $\log \phi_{\cX^*}(t)$.  Given the many
different characterizations of CGFs introduced in
\Cref{thmdef:cgf_equivalences}, we have a rich set of tools for
studying these complex functions.  In particular, we will show
that the limiting standardized characteristic functions are
entire, see \Cref{cor:entire_limits}.  This allows us to complete the proof of
\Cref{thm:FS_iff}.  We begin by spelling out the connections between
the CGFs in rational and cyclotomic form with the first and second
(standardized) characteristic functions.

Recall from \Cref{sub:prob.and.cum} that for any CGF $f(q)$ with
corresponding random variable $\cX$, we have $\bE[q^{\cX}] =
f(q)/f(1)$, so the characteristic function of 		   $\cX$ is 
\[
\phi_{\cX}(t) \coloneqq \bE[e^{it\cX}] = f(e^{it})/f(1), 
\]
and the cumulant generating function of $\cX$ is 
\[ \log \phi_{\cX}(t) = \log \bE[e^{it\cX}] = \sum_{d=1}^\infty
\kappa_d(f) \frac{(it)^d}{d!}.
\]
Furthermore, $\kappa_d(f) = 0$ for $d>1$ odd, $\mu =\kappa_1(f)$ is the
mean of $\cX$ and $\sigma^{2}=\kappa_2(f)$ is its variance.

Note, both $f(q)$ and $\alpha q^{\beta}f(q)$ for positive integers
$\alpha, \beta$
give rise to the same standardized random variable $\cX^*$ so in order
to study all standardized CGF distributions, it suffices to assume
$f(q) \in \Phi^{+}$.  Thus, if
$f(q)=\Phi_{i_{1}}(q)\Phi_{i_{2}}(q)\cdots \Phi_{i_{k}}(q) = \prod_{j=1}^k
\frac{[a_j]_q}{[b_j]_q}$, 
then we can extend the corresponding characteristic function to the
complex plane by setting 
\begin{equation}\label{eq:char.poly.in.f}
\phi_{\cX}(z) = f(e^{iz})/f(1)= \Phi_{i_{1}}(e^{iz})\Phi_{i_{2}}(e^{iz})\cdots
\Phi_{i_{k}}(e^{iz})/f(1)
= \prod_{j=1}^k \frac{[a_j]_{e^{iz}}}{[b_j]_{e^{iz}}}
\end{equation}
and 
\begin{align}\label{eq:std.char.poly}
\phi_{\cX^*}(z) & =\bE[e^{iz\cX^*}] = \bE[e^{iz(X-\mu)/\sigma}] =
e^{-iz\mu/\sigma} \bE[e^{izX/\sigma}] \\
&= e^{-iz\mu/\sigma} f(e^{iz/\sigma})/f(1) = e^{-iz\mu/\sigma}
\phi_{\cX}(z/\sigma ). \nonumber
\end{align}
Furthermore, we have a convergent power series
representation of the cumulant generating function which can be
expressed in terms of the multisets in the rational form of $f(q)$ as
\begin{align}\label{eq:second.char.function}
\log \phi_{\cX^*}(z) &= \sum_{d=1}^\infty (-1)^{d}\kappa_{2d}^{*}(f)
\frac{z^{2d}}{(2d)!} \\
&= \sum_{d=1}^\infty
(-1)^{d}\left(\frac{B_{2d}\sum_{k=1}^m (a_k^{2d} - b_k^{2d})}{2d((B_{2}/2) \sum_{k=1}^m (a_k^2 - b_k^2))^{d}}  \right)
\frac{z^{2d}}{(2d)!} \nonumber
\end{align}
since $\kappa_{2d}^{*}(f) = \kappa_{2d}(f)/\kappa_{2}(f)^{d}$ and
$\kappa_{d}(f)$ is given by \eqref{eq:prod_cumulants} for all positive
integers $d$.  Furthermore, $(-1)^{d}\kappa_{2d}^{*}(f)$ is
a non-positive real number by \Cref{cor:kappa_alternates}.

\subsection{CGF characteristic functions}

Consider the set of standardized characteristic functions of random
variables associated to cyclotomic generating functions.   As
mentioned above, it suffices to consider only basic CGFs, so consider
the set of all standardized CGF characteristic functions on the
complex plane, 
\[
\cC_{\CGF} \coloneqq \{\phi_{\cX^*}(z) : \bE[q^{\cX}] = f(q)/f(1)\text{ for }f(q) \in \Phi^{+} \}.
\]
We first show that $\cC_{\CGF}$ is a \textbf{normal family} of
continuous functions.  This means that every infinite sequence
$(\phi_{\cX_{N}^*}(z): N=1,2,\dots )$ in $\cC_{\CGF}$ contains a
subsequence which converges uniformly on compact subsets of
$\mathbb{C}$.

\begin{theorem}\label{thm:normal}
The set $\cC_{\CGF}$ is a normal family of entire functions.

\begin{proof}
Given $\phi_{\cX^*}(z) \in \cC_{\CGF}$, let $f(q) \in \Phi^{+}$ be the
associated CGF, and let $\cX$ be the corresponding random variable with
$\bE[q^{\cX}] = f(q)/f(1)$, mean $\mu$ and standard deviation
$\sigma$.  Let $\tilde{f}(q) = q^{-\mu} f(q)/f(1) \in \bR_{\geq
0}[q^{\pm 1/2}]$. Then, by \eqref{eq:std.char.poly} 
\begin{equation}\label{eq:in.term.f}
\phi_{\cX^*}(z)  = e^{-iz\mu/\sigma}
f(e^{iz/\sigma})/f(1) = \tilde{f}(e^{iz/\sigma}).
\end{equation}
Thus, since $f(q) \in \Phi^{+}$ is a finite product of cyclotomic
polynomials, $\phi_{\cX^*}(z)$ is a finite sum of products of
exponential functions by \eqref{eq:char.poly.in.f} and \eqref{eq:std.char.poly}, hence it is entire.

From Montel's theorem in complex analysis, $\cC_{\CGF}$ is a normal
family of entire functions if and only if it is bounded on all complex
disks $|z| \leq R$. Hence we will bound $|\phi_{\cX^*}(z)|$ in terms
of $R$.  Note that $f(q)=1$ if and only if $\sigma =0$ by
\Cref{lem:cgf_ineqs}(iii), in which case $\phi_{\cX^*}(z)=1$ is
bounded by any function greater than 1.  So, we will assume $f(q)$ is
not constant and $\sigma >0$.  By \cite[Lem.~2.8]{HwangZacharovas},
for all real $t \geq 0$,\footnote{\cite[Lem.~2.8]{HwangZacharovas} is
incorrectly stated for all $t \in \bR$, though one of the last steps
in the argument requires $t \geq 0$.}

\begin{equation}\label{eq:hz.bound}
 \bE[e^{t\cX^*}] \leq \exp\left(\frac{3}{2} t^2 e^{2t/\sigma}\right).
\end{equation}
Since $\bE[e^{z \cX^*}] = \phi_{\cX^*}(-iz)$, we can use this
inequality to bound $|\phi_{\cX^*}(z)|$ for all complex $|z| \leq R$
as follows.

For all $|z| \leq R$, we claim that $|\tilde{f}(e^{iz/\sigma})| \leq
2\tilde{f}(e^{R/\sigma})$. Indeed, since $f(q) \in \Phi^{+}$, we have  $\tilde{f}(q) = \sum_{k=-N}^N
a_k q^{k/2}$ for $a_k \geq 0$ satisfying $a_k = a_{-k}$. Therefore,
\begin{align*}
  \left|\tilde{f}(e^{iz/\sigma})\right|
    &= \left|\sum_{k=-N}^N a_k e^{izk/2\sigma}\right|
    \leq \sum_k a_k \left|e^{izk/2\sigma}\right|
    = \sum_k a_k e^{\Re(izk/2\sigma)} \\
    &\leq \sum_k a_k e^{R|k|/2\sigma}
    \leq \sum_k a_k (e^{Rk/2\sigma} + e^{-Rk/2\sigma})
    = \sum_k a_k e^{Rk/2\sigma} + \sum_k a_{-k} e^{-Rk/2\sigma} \\
    &= 2\tilde{f}(e^{R/\sigma}).
\end{align*}
By \Cref{lem:cgf_ineqs}(iii), $\sigma^2 \geq 1/4$ is bounded away from
$0$ since we assumed $f(q)$ is not constant.  Therefore, by the claim,
\eqref{eq:in.term.f}, and \eqref{eq:hz.bound}, we have the required
uniform bound \[ |\phi_{\cX^*}(z)| = |\tilde{f}(e^{iz/\sigma})| \leq
2\tilde{f}(e^{R/\sigma}) = 2\phi_{\cX^*}(-iR) = 2\bE[e^{R\cX^*}] \leq
2\exp\left(\frac{3}{2} R^2 e^{8R}\right). \]

  \end{proof}
\end{theorem}

\begin{theorem}[Converse of Frech\'et--Shohat for CGFs]\label{thm:frechet_shohat_converse}
  Suppose $\cX_1, \cX_2, \ldots$ is a sequence of random variables
  corresponding to cyclotomic generating functions such that
  $\cX_n^* \Rightarrow \cX$ for some random variable $\cX$.
  Then $\cX$ is determined by its moments and, for all $d \in \bZ_{\geq 1}$,
    \[ \lim_{n \to \infty} \mu_d^{\cX_n^*} = \mu_d^{\cX}. \]

  \begin{proof} By \Cref{thm:normal}, we may pass to a subsequence and
assume $\phi_{\cX_n^*}(z)$ converges uniformly on compact subsets of
$\bC$.  Hence, they converge to an entire function $\phi(z)$. On the
other hand, by L\'evy continuity, $\phi_{\cX_n^*}(t) \to
\phi_{\cX}(t)$ pointwise for all $t \in \bR$.  Thus, $\phi_{\cX}(t) =
\phi(t)$ for $t \in \bR$.

A priori, a characteristic function $\phi_{\cX}(t) \coloneqq
\bE[e^{it\cX}]$ exists only for $t \in \bR$. However, in this
case we have an entire function $\phi(z)$ which coincides with
$\phi_{\cX}(z)$ for all $z \in \bR$. By \cite[Thm.~7.1.1,
pp.191-193]{MR0346874}, agreement on the real line suffices to show
$\phi_{\cX}(z)$ exists and agrees with $\phi(z)$ for all $z \in
\bC$.  Since $\phi(z)$ is entire, it can be represented by a power
series that converges everywhere in the complex plane.  Hence, the
moment-generating function $\bE[e^{t\cX}]=\phi_{\cX}(-it)=\phi(-it)$ exists for all $t \in
\bR$ and it can be expressed as a convergent power series with finite
coefficients.  Hence $\cX$ has moments of all orders and is determined
by its moments by \cite[Thm.~30.1]{MR1324786}.
\end{proof}
\end{theorem}

\begin{corollary}\label{cor:entire_limits}
  If $\cX_1, \cX_2, \ldots$ is a sequence of CGF random variables and
  $\cX_n^* \Rightarrow \cX$, then $\phi_{\cX}(z)$ is entire.
\end{corollary}

We may now prove \Cref{thm:FS_iff} from the introduction.

\begin{proof}[Proof of {\Cref{thm:FS_iff}}] Suppose $\cX_N^*
\Rightarrow \cX$ converges in distribution. Then the result follows by
\Cref{thm:frechet_shohat_converse}. Conversely, suppose $\lim_{N \to
\infty} \mu_d^{\cX_N^*}$ exists and is finite for all $d \in \bZ_{\geq
1}$. By \Cref{thm:normal}, we may pass to a subsequence on which
$\phi_{\cX_N^*}(z)$ converges uniformly on compact subsets of
$\bC$. The limiting entire function $\phi(z)$ then has power series
coefficients determined by the $\mu_{d}:=\lim_{N \to \infty} \mu_d^{\cX_N^*}$,
so the limit is independent of the subsequence we passed to and
$\lim_{N \to \infty} \phi_{\cX_N^*}(t) = \phi(t)$ for all $t \in
\bR$. Moreover, $\phi(t)$ is continuous at $0$, so there exists $\cX$
such that $\cX_N^* \Rightarrow \cX$ by \cite[Cor.~26.1]{MR1324786}.
\end{proof}

\subsection{Formal cumulants and cumulant generating  functions}

Recall formal cumulants from \Cref{ssec:formal.cum}.
The formal cumulants of cyclotomic polynomials have the following explicit  growth rate.

\begin{lemma}\label{lem:formal.cums.cyc.ineq}
For each fixed  $d\geq 1$, the formal cumulants of the cyclotomic polynomials for $n > 1$ 
satisfy 
\begin{equation}\label{eq:cyclotomic.cum.2}
\frac{(2d)!}{d} \left(\frac{n}{2\pi}\right)^{2d} \leq |\kappa_{2d}|
\leq \frac{|B_{2d}| n^{2d}}{2d}.
\end{equation}
 
\begin{proof}     
    The upper bound in \eqref{eq:cyclotomic.cum.2} is clear from the
factored form of the Jordan totient formula in
\eqref{eq:cyclotomic.cum.1}. For the lower bound, the Euler product
for the Riemann zeta function is \[ \prod_{\text{$p$ prime}}
(1-p^{-s}) = \frac{1}{\zeta(s)}. \] See \cite{enwiki:1185992088} for a
sketch of the proof.  Applying the expression for the zeta function at
positive even integers gives \begin{align*}
\frac{|\kappa_{2d}|}{n^{2d}}
        &= \frac{|B_{2d}|}{2d} \prod_{\substack{p\text{ prime} \\ p \mid n}} (1 - p^{-2d}) \\
        &\geq \frac{|B_{2d}|}{2d} \frac{1}{\zeta(2d)} \\
        &= \frac{|B_{2d}|}{2d} \frac{2(2d)!}{(2\pi)^{2d} |B_{2d}|} \\
        &= \frac{1}{d} \frac{(2d)!}{(2\pi)^{2d}}.
    \end{align*}
  \end{proof}
\end{lemma}

While cyclotomic generating functions are typically given in rational
form, \Cref{lem:formal.cums.cyc.ps} allows their cumulants to be
described in terms of the cyclotomic form as follows.  One may for
instance use \Cref{ex:hooks} to describe the $q$-hook formula in
cyclotomic form combinatorially.

\begin{corollary}\label{cor:cums.in.cyclotmic.form} If $f(q) =
\prod_{j=1}^k \Phi_{n_j}(q)$, then for each $d \geq 1$ we have
\[
\kappa_{2d}(f(q)) = \sum_{j=1}^{k} \kappa_{2d}(\Phi_{n_j}(q)) =  \sum_{j=1}^{k}\frac{B_{2d}}{2d} n_{j}^{2d} \prod_{\substack{\text{$p$
prime} \\ p \mid n_{j}}} \left(1 - \frac{1}{p^{2d}}\right)
\]
and uniformly 
    \[ |\kappa_{2d}(f(q))| = \Theta\left(\sum_{j=1}^k n_j^{2d}\right). \]
\end{corollary}

The asymptotic behavior of sequences of CGF random variables can be
determined by their standardized cumulant generating function.  For
instance, $\cX_1,\cX_2,\ldots$ is asymptotically normal if and only if
the corresponding sequence of standardized cumulants
$(\kappa_d^{(N)})^* \to 0$ for all $d\geq 3$ as $N \to \infty,$ in
which case by \eqref{eq:second.char.function}, $\log \phi_{\cX_N^*}(t)
\to -t^2/2$ for all $t \in \bR$ as $N\to \infty$. Standardized second
characteristic functions of CGFs are particularly well-behaved in
the following sense.

\begin{proposition}
  Let $\cX$ be a random variable corresponding to a cyclotomic generating function. Then
  the standardized cumulant generating function $\log
  \phi_{\cX^*}(z)$ is analytic in the complex disk $|z| <
  \sqrt{2}$. Furthermore, for real values $-\sqrt{2} < t < \sqrt{2}$ we
  have 
  \begin{equation}\label{eq:log_phi}
    \log \phi_{\cX^*}(t) \leq -\frac{t^2}{2} = \log \phi_{\cN(0,1)}(t).
  \end{equation}
  
  \begin{proof} By \Cref{thm:normal}, $\phi_{\cX}(z)$ is entire, so
  the singularities of $\log \phi_{\cX}(z)$ come about from the zeros of
$\phi_{\cX}(z)$.  To determine the zeros, it suffices to assume $\cX$
is determined by $f(q) = \frac{\prod_{k=1}^m [a_k]_q}{\prod_{k=1}^m
[b_k]_q} \in \Phi^{+}$ since the results pertain to the standardized
random variable corresponding to $\cX$.  Furthermore, since
$\phi_{\cX}(z)=1$ if $f(q)=1$, which satisfies both claims, we may
assume that $m = \max\{a_k\}>1$ and $m$ does not appear in the
denominator multiset.  Hence, $\Phi_{m}(q)$ is a factor of $f(q)$.
Furthermore, from the product formula \eqref{eq:cgf_cf}, we observe
that the closest zero of $\phi_{\cX}(z)=f(e^{iz})$ to the origin
occurs at $z=2\pi/ m$. Hence $\phi_{\cX}(z)$ is analytic on the
simply connected domain $|z| < 2\pi/m$, so $\log \phi_{\cX}(z)$
is analytic on the disk $|z| < 2\pi/m$.

By \eqref{eq:std.char.poly}, $\phi_{\cX^*}(z) = e^{-i\mu z/\sigma}\phi_{\cX}(z/\sigma)$, where $\mu = \kappa_1(f)$ and $\sigma^{2}= \kappa_2(f)$ are
the mean and variance of $\cX$.  Therefore, to prove $\log
\phi_{\cX^*}(z)$ is analytic in the complex disk $|z| < \sqrt{2}$, it
suffices to show that $|z| < \sqrt{2}$ implies
$\left|\frac{z}{\sigma}\right| < 2\pi/m.$

Since $\Phi_{m}(q)$ is a factor of $f(q)$, we have $0<
\kappa_2(\Phi_m) \leq \sigma^2$ by \Cref{lem:formal.cums.cyc.ineq} and
\Cref{cor:cums.in.cyclotmic.form}.  By the lower bound in
\eqref{eq:cyclotomic.cum.2}, $\sqrt{ \kappa_2(\Phi_m)} \geq
\frac{m}{\pi \sqrt{2}}$.  Hence if $|z| < \sqrt{2}$,
\[ \left|\frac{z}{\sigma}\right| < \frac{\sqrt{2}}{m/\pi\sqrt{2}} =
\frac{2\pi}{m} 
\]
as required.  

To prove the ``furthermore'' statement,  we use the expansion
    \begin{equation}\label{eq:log_phi_star}
      \log\phi_{\cX^*}(z) = -\sum_{d \geq 1} \frac{(-1)^{d-1} \kappa_{2d}}{(2d)!}
         \left(\frac{z}{\sigma}\right)^{2d}
    \end{equation}
to compute $\log\phi_{\cX^*}(t)$ for real $-\sqrt{2} < t < \sqrt{2}$.
The inequality
\eqref{eq:log_phi} comes from truncating this power series expansion
after the first term and applying \Cref{cor:kappa_alternates} to see
that all terms on the right side are nonpositive.
\end{proof}
\end{proposition}

\section{CGF monoids and related open problems}\label{sec:examples}

As mentioned in Definition~\ref{def:basic}, the set of basic cyclotomic
generating functions forms a monoid under multiplication, $\Phi^+$.
We also consider the larger monoid generated by all cyclotomic
polynomials.  Since many families of polynomials of interest are also
either unimodal or log-concave with no internal zeros, we consider
these submonoids as well, together with another variant using the Gale
order on multisets.  We conclude with a monoid associated to Hilbert
series of polynomial rings quotiented by regular sequences and several
open problems.

\subsection{Basic, unimodal, and log-concave CGF monoids}\label{sec:monoid} %

\begin{definition}
  The \textbf{cyclotomic monoid} is the monoid $\Phi^{\pm}$ generated
by the cyclotomic polynomials under multiplication, graded by
polynomial degree.  The polynomials in $\Phi^{\pm}$ can have both
positive and negative integer coefficients.  Recall that the basic CGF
monoid is the submonoid $\Phi^+$ of $\Phi^\pm$ consisting of all
basic cyclotomic generating functions, which is clearly closed under
multiplication.

Similarly, let $\Phi^{\uni}$ and $\Phi^{\lcc}$ denote the submonoids
of $\Phi^+$ given by cyclotomic generating functions which are
unimodal or log-concave with no internal zeros, respectively.  These
properties are preserved under multiplication
\cite[Prop.~1-2]{MR1110850}.
\end{definition}

\begin{lemma}\label{lem:monoid.properties}
We have $\Phi^{\lcc} \subset \Phi^{\uni} \subset \Phi^+ \subset
\Phi^{\pm}$, and for each monoid $\cM$ in $ \{\Phi^{\lcc}, \Phi^{\uni},$
 $\Phi^+, \Phi^{\pm}\}$ the following facts hold.

  \begin{enumerate}[(i)]
    \item Each $\cM$ has a unique minimal set of generators under inclusion,
      namely the set of irreducible elements $x \in \cM$
      where $x=yz$ for $y, z \in \cM$ implies $y=1$ or
      $z = 1$.
    \item There are only a finite number of polynomials in $\cM$ of
    any given degree $n$.  
    \item Each $\cM$ cannot be generated by a finite subset.
    \item Any element $1 \neq f \in \Phi^+$ has a cyclotomic
      factor of the form $\Phi_{p^k}(q)$ for some prime $p$.
    \item A cyclotomic polynomial $\Phi_n(q)$ is in
      $\Phi^+$ if and only if $n$ is a prime power.
    \item Any  $f \in \Phi^+$ with odd degree has $\Phi_{2}(q)$ as a factor.
  \end{enumerate}

  \begin{proof} The list of inclusions follows directly from the
definitions and the fact that log-concave polynomials with no internal
zeros are always unimodal.  Property (i) follows easily from classical
factorization in $\bC[q]$ and the fact that every polynomial in
$\Phi^\pm$ is monic.  For (ii), the $\cM = \Phi^\pm$ case follows from
the fact that Euler's totient function has finite fibers, and the rest
are submonoids.  For (iii), note that $[n]_q \in \cM$ for each such
$\cM$, and $[n]_q$ has primitive $n$th roots of unity as roots, so as
$n \to \infty$ a finite set of generators cannot yield all $[n]_q$.
    
    For (iv), first note that $f(1) \geq 1 \neq 0$ for $f \in \Phi^{+}$, so
    $\Phi_1(q) = q-1$ is not a factor of $f$.
    Further recall that $\Phi_n(0) = \Phi_n(1) = 1$ for $n \geq 2$ not a prime power.
    Hence, if $f$ had no factors of the form $\Phi_{p^k}(q)$, we would have
    $f(0) = f(1) = 1$, forcing $f(q) = 1$ since the coefficients of $f$ are nonnegative.
    The forwards implication in (v) is given by (iv), and the backwards implication follows
    from the fact that
      \[ \Phi_{p^k}(q) = 1 + q^{p^k} + q^{2p^k} + \cdots +
      q^{(p-1)p^k} \in \Phi^+. \]
 Finally, for (vi), $\varphi(n)$ is even for all $n \geq 3$, so
    any product of cyclotomic polynomials of odd degree must contain $\Phi_2$.
  \end{proof}
\end{lemma}

\begin{example}
  We have
  \begin{align*}
  \Phi_5(q) \Phi_6(q)
    &= (q^4+q^3+q^2+q+1)(q^2-q+1) \\
    &= q^6 + 0q^5 + q^4 + q^3 + q^2 + 0q^1 + 1 \in \Phi^+ -
\Phi^{\uni}.  \end{align*} Consequently, $\Phi_5(q)\Phi_6(q)$ is
irreducible in $\Phi^+$ and hence belongs to its minimal set of
generators. The smallest degree polynomial in $\Phi^{\uni} -
\Phi^{\lcc}$ is
\[
  \Phi_3(q) \Phi_4(q) = q^{4} + q^{3} + 2q^{2} + q^{1} + 1.   
\]

\end{example}

Intuitively, multiplying by $\Phi_p(q) = 1 + q + \cdots + q^{p-1}$ for
$p$ prime tends to smooth out chaotic coefficients. Computationally,
it appears to be necessary to include such a factor in unimodal or
log-concave CGFs. More precisely, we conjecture the following
analogue of \Cref{lem:monoid.properties}(iv), which has been checked up to degree $50$.

\begin{conjecture}
  Any element $1 \neq f \in \Phi^{\uni}$ has a cyclotomic factor of the form $\Phi_p(q)$ for some prime $p$.
\end{conjecture}

In 1988, Boyd--Montgomery \cite{bm88} described some of the properties
of the cyclotomic monoid $\Phi_n^{\pm}$.  In particular, they showed
that $\log|\Phi_n^{\pm}| \sim \sqrt{105 \zeta(3) n}/\pi$.  This result
is attributed to Vaclav Kotesovec in \cite[A120963]{oeis}, but no
citation is given.  One could ask for similar results for the other
cyclotomic monoids.

\begin{problem}\label{prob:growth_rate}
For a monoid of polynomials $\cM$, let $\cM_n$ be the set of degree
$n$ polynomials in $\cM$.  Identify the growth rate of $|\cM_n|$
as $n \to \infty$ for the other cyclotomic monoids $\cM \in
\{\Phi^{\lcc}, \Phi^{\uni}\}$.  
\end{problem}

\begin{problem}\label{prob:mins}
  Classify the minimal generating set $\cM \in
\{\Phi^{\lcc}, \Phi^{\uni}, \Phi^+\}$. Give an efficient algorithm for
identifying the generators up to any desired degree. Find the
asymptotic growth rate of the generating set as a function of degree.
\end{problem}

We use the cyclotomic form for basic CGFs and the fact that we know a
lower bound for the degree of the cyclotomic polynomial $\Phi_{n}(q)$
to find all basic CGFs of a given degree.  Some initial data and OEIS
identifiers \cite{oeis} can be found in \Cref{s:appendix}.  The
sequence corresponding to $\Phi_n^{\pm}$ \cite[A120963]{oeis} was
created in 2006.  We added the corresponding pages for the other
cyclotomic monoids, so we believe the other cyclotomic monoids have
not been studied in the literature.

\subsection{Gale order and the associated CGF monoid}\label{sub:gale}

Given two multisets of the same size, say $A=\{a_{1}\leq a_{2}\leq
\cdots \leq a_{m} \}$ and $B=\{b_{1}\leq b_{2}\leq \cdots \leq b_{m}
\}$ both sorted into increasing order, we say $A\preceq B$ in
\textbf{Gale order} provided $a_{k}\leq b_{k}$ for all $1\leq k\leq
m$.  This partial order is known by many other names; we are following
\cite{Ardila-Rincon-Williams} for consistency. Gale studied this
partial order in the context of matroids on $m$-subsets of
$\{1,2,\dots,n \}$ in the 1960's \cite{Gale.1968}.

\begin{definition}\label{def:gale.monoid}
Let $\Phi^{\gale}$ denote the \textbf{Gale} submonoid of $\Phi^+$
given by cyclotomic generating functions $f(q)=\frac{\prod_{k=1}^m
[a_k]_q}{\prod_{k=1}^m [b_k]_q} \in \Phi^{+}$ such that $\{b_{1},
b_{2},\ldots ,b_{m}\} \preceq \{a_{1}, a_{2}, \ldots , a_{m} \}$ in
Gale order.  Note that Gale order holds independent of the representation
of the rational expression chosen since $\{b_{1}, b_{2},\ldots
,b_{m}\} \preceq \{a_{1}, a_{2}, \ldots , a_{m} \}$ if and only if
$\{i, b_{1}, b_{2},\ldots ,b_{m}\} \preceq \{i, a_{1}, a_{2}, \ldots ,
a_{m} \}$ for any positive integer $i$.  Furthermore, the Gale
property is again preserved under multiplication, hence $\Phi^{\gale}$
is closed under multiplication.
\end{definition}

The properties in \Cref{lem:monoid.properties} also hold for the Gale
monoid.  In particular, it has a finite number of elements of each
degree and a unique minimal set of generators which can be explored
computationally.  Data is given in \Cref{s:appendix} for the number of
elements in $\Phi^{\gale}$ of degree $n$ up to $n=18$ along with the
pointer to the corresponding OEIS entry.  The number of generators of
each degree are also noted in \Cref{s:appendix}.

As noted in \Cref{rem:cgf_not_pointwise}, not all basic
CGFs are in the Gale monoid.  They agree up to degree 10, there are
two basic CGFs which don't satisfy the Gale property in degree 11,
namely

\begin{align*}
 q^{11} + & 4q^{10} + 8q^{9} + 9q^{8} + 5q^{7} + 5q^{4} + 9q^{3} + 8q^{2} + 4q^{1} + 1
 = \frac{[12]_q[3]_q^{3}[2]_q^{2}}{[6]_q[4]_q[1]_q^{6}} = \Phi_{12} \Phi_3^{3}  \Phi_2 \\
q^{11} + & 4q^{10} +  5q^{9} + q^{8} + 5q^{6} + 5q^{5} + q^{3} + 5q^{2} + 4q^{1} + 1
 = \frac{[12]_q[2]_q^{5}}{[4]_q[3]_q[1]_q^{4}} = \Phi_{12} \Phi_6 \Phi_2^{5} .
\end{align*}
There are 4 non-Gale basic CGFs in degree 12, and so on.

The Gale monoid and unimodal CGF monoids are not comparable.  The Gale
monoid includes $[4]_{q}/[2]_{q}=\Phi_{4}=1+q^{2}$, which is not
unimodal.  The smallest degree unimodal CGFs which don't satisfy the
Gale property have degree 20,

\begin{align*}
f(q) &= \frac{[12]_q[8]_q[3]_q[3]_q[3]_q[3]_q[3]_q[3]_q[2]_q}{[6]_q[4]_q[4]_q[1]_q[1]_q[1]_q[1]_q[1]_q[1]_q} = \Phi_{12} \Phi_8 \Phi_3^{6} \\
g(q) &= \frac{[12]_q[8]_q[3]_q[3]_q[3]_q[3]_q[3]_q[2]_q[2]_q[2]_q}{[6]_q[4]_q[4]_q[1]_q[1]_q[1]_q[1]_q[1]_q[1]_q[1]_q} = \Phi_{12} \Phi_8 \Phi_3^{5} \Phi_2^{2}.
\end{align*}

The Gale monoid also does not contain the log-concave monoid.  The
unique smallest degree log-concave CGF which does not satisfy the Gale
property has degree 25,

\[
\frac{[12]_q[2]_q^{19}}{[4]_q[3]_q} =  \Phi_{12} \Phi_6 \Phi_2^{19}.
\]

\subsection{CGF monoid from regular sequences}\label{sub:hsop}

In this subsection, we discuss one more submonoid of the
basic CGF monoid coming from certain Hilbert series of quotients of
polynomial rings that naturally arise in commutative algebra.  As
described below, this monoid is also a submonoid of the Gale monoid.

Fix a pair of multisets of positive integers of the same size, say
$\{a_1, \ldots, a_m\}$ and $\{b_1, \ldots, b_m\}$.  Let $\bk[x_1,
\ldots, x_m]$ be a (free) polynomial ring over a field $\bk$ with
grading determined by $\deg(x_j) = b_j$ for $1\leq j\leq m$.  Let
$\theta_1, \ldots, \theta_m$ be a sequence of non-constant homogeneous
polynomials in $\bk[x_1, \ldots, x_m]$ with $\deg(\theta_j) = a_j $.
Then, $\theta_1, \ldots, \theta_m$ is a \textbf{regular sequence} if
$\theta_i$ is not a zero-divisor in $\bk[x_1, \ldots, x_m]/(\theta_1,
\ldots, \theta_{i-1})$ for all $1 \leq i \leq m$.  Furthermore,
$\theta_1, \ldots, \theta_m \in \bk[x_1, \ldots, x_n]$ form a
\textbf{homogeneous system of parameters} (HSOP) if $\bk[x_1, \ldots, x_m]$ is
a finitely generated $\bk[\theta_1, \ldots, \theta_m]$-module.  Consider
the corresponding quotient rings,
\[
R \coloneqq \bk[x_1, \ldots, x_m]/(\theta_1, \ldots, \theta_m).
\]
Such rings play an important role in commutative algebra and the study
of affine and projective varieties \cite{Smith.etal}.  The following
equivalences are well-known and are stated explicitly for
completeness.

\begin{lemma}\label{lem:hsop.eqivs}
  Let $\bk[x_1, \ldots, x_m]$ be a polynomial ring over a field $\bk$ with $\deg(x_j) = b_j \in \bZ_{\geq 1}$. Suppose $\theta_1, \ldots, \theta_m$ are homogeneous elements. Then the following are equivalent:
  \begin{enumerate}[(i)]
    \item $\theta_1, \ldots, \theta_m$ is a regular sequence;
    \item $\theta_1, \ldots, \theta_m$ is a homogeneous system of parameters;
    \item $R$ is finite-dimensional over $\bk$;
    \item for all $1 \leq i \leq m$, there is some $N$ such that $x_i^N \in (\theta_1, \ldots, \theta_m)$;
    \item $R$ has Krull dimension $0$; and
    \item the only common zero of $\theta_1, \ldots, \theta_m$ over the algebraic closure $\overline{\bk}$ is the origin.
  \end{enumerate}

  \begin{proof}
    For (i)-(v), see for instance \cite[\S3]{MR0485835} and
    \cite[\S3]{stanley.1979}, and \cite{Cox-Little-OShea} for the equivalence of
    (iii)-(vi). In brief, (ii) $\Leftrightarrow$ (iii) is elementary, and (i) $\Leftrightarrow$ (ii) since $\bk[x_1, \ldots, x_m]$ is \textit{Cohen--Macaulay}, so regular sequences and homogeneous systems of parameters coincide. The equivalence (iii) $\Leftrightarrow$ (iv) is also elementary. For (iv) $\Leftrightarrow$ (v), the Krull dimension is the order of the pole of the Hilbert series of $R$ at $1$, which is $0$ if and only if $R$ is finite-dimensional over $\bk$. For (vi), the Krull (and vector space) dimension of $R$ is preserved by extension of scalars, and in the algebraically closed case we may apply the Nullstellensatz.
  \end{proof}
\end{lemma}

Let $R_{k}$ be the linear span of the homogeneous degree $k$ elements
in $R$.  Then the \textbf{Hilbert series} of $R$ is
\[
\Hilb(R; q) \coloneqq \sum_{k \geq 0} (\dim R_k) q^k
\]
If $R$ is finite-dimensional over $\bk$, then $\Hilb(R; q)$ is a
polynomial.

Given a homogeneous system of parameters $\theta_1, \ldots, \theta_m$
with corresponding ring $R$, is such a Hilbert series given by a CGF?
The affirmative answer given in the following theorem is based on the
work of Macaulay \cite{MR1281612} from 100 years ago. The proof uses
the natural short exact sequence $0 \to (\theta) \to R \to R/(\theta)
\to 0$ and induction.  Stanley built on this theory in his
work on Hilbert functions of graded algebras \cite{MR0485835} with
connections to invariant theory.

\begin{theorem}[{c.f.~\cite[Cor.~3.2-3.3]{MR0485835}}]\label{thm:cgf_tensor}
  If $\theta_1, \ldots, \theta_m$ is a regular sequence with
  $\deg(\theta_j) = a_j$ in the polynomial ring $\bk[x_1, \ldots,
  x_m]$ with $\deg(x_j) = b_j$, then
\[
R=\bk[x_1, \ldots, x_m]/(\theta_1, \ldots, \theta_m)
\]
  is a finite-dimensional $\bk$-vector space, and its Hilbert series is the basic cyclotomic generating function
  \begin{equation}\label{eq:cgf_tensor}
    \Hilb(R; q) \coloneqq \sum_{k \geq 0} (\dim R_k) q^k
      = \prod_{k=1}^m \frac{[a_k]_q}{[b_k]_q} \in \Phi^{+}.
  \end{equation}
  Moreover, the converse holds: if \eqref{eq:cgf_tensor} holds, then $\theta_1, \ldots, \theta_m$ is
  a regular sequence.
\end{theorem}

\begin{example}
  Consider the ring \begin{equation}\label{eq:gr_CH} \frac{\bC[x_1,
\ldots, x_k]^{S_k}}{\langle h_{\ell+1}, \ldots,
h_{\ell+k}\rangle} \end{equation} where $h_i$ denotes the $i$th
homogeneous symmetric polynomial in $k$ variables $x_1, \ldots, x_k$.
This ring is well-known to be a presentation of the cohomology ring of
the complex Grassmannian of $k$-planes in $\ell+k$-space. The
polynomials $h_{\ell+1}, \ldots, h_{\ell+k}$ form a homogeneous system
of parameters, so we have the corresponding cyclotomic generating
function \[ \frac{[\ell+1]_q \cdots [\ell+k]_q}{[1]_q \cdots [k]_q} =
\binom{\ell+k}{k}_q, \] recovering the $q$-binomial coefficients.  See
\cite[Prop.~2.9]{MR2542141} for more regular sequences of a similar flavor.
\end{example}

\begin{example}
  Bessis introduced an HSOP associated to a
  well-generated complex reflection groups $W$ \cite[Thm.~5.3]{MR3296817}.
  The corresponding Hilbert series is a $q$-analogue of the $W$-noncrossing
  partition number,
    \[ \prod_{i=1}^n \frac{[ih]_q}{[d_i]_q} \in \Phi^{\alg}. \]
  CGFs are often intimately related to the cyclic sieving phenomenon
  \cite{Reiner-Stanton-White.CSP}, and Douvropolous proved a CSP
  in this context \cite{MR3940661}.
\end{example}

Suppose two basic CGFs $f, g$ arise as the Hilbert series of quotient
rings $R_1$ and $R_2$ coming from homogeneous systems of
parameters. Then $fg$ arises in the same way by taking the tensor
product of $R_1$ and $R_2$. Hence we may consider the \textbf{HSOP
monoid} $\Phi^{\alg}$ as a submonoid of $\Phi^+$ consisting of all
Hilbert series of quotients $\bC[x_1, \ldots, x_m]/(\theta_1, \ldots,
\theta_m)$ for homogeneous systems of parameters.   

\begin{lemma}  The HSOP monoid $\Phi^{\alg}$ is a submonoid of the Gale monoid $\Phi^{\gale}$.
  
  \begin{proof}

  Suppose $\theta_1, \ldots, \theta_m$ is a regular sequence for
$\bk[x_1, \ldots, x_m]$ where $\deg(\theta_i) = a_i$, $\deg(x_i) = b_i$,
and we have sorted the elements so that $a_1 \leq \cdots \leq a_m$ and
$b_1 \leq \cdots \leq b_m$. If $a_k < b_k$, then $\theta_1, \ldots,
\theta_k$ have degree at most $a_k$, and $\deg(x_k) = b_k > a_k$, so
$\theta_1, \ldots, \theta_k \in \bk[x_1, \ldots, x_{k-1}]$.
By \Cref{thm:cgf_tensor}, $\bk[x_1, \ldots, x_{k-1}]/(\theta_1, \ldots,
\theta_{k-1})$ is finite-dimensional, so it contains a power of
$\theta_k$, contradicting the zero-divisor condition.  Hence, $a_k
\geq b_k$ for all $k=1, \ldots, m$ and the result follows from the
definition of Gale order.
\end{proof}
\end{lemma}

Which basic CGFs can be realized as the Hilbert series of a quotient
ring by a homogeneous sequence of parameters?  In the special case
$b_1=\cdots=b_m=1$, i.e.~$\deg(x_j) = 1$ for all $j$, a homogeneous
system of parameters corresponds to a sequence of generators for a
classical \textbf{complete intersection} $X$. The corresponding
cyclotomic generating function is the Hilbert series of the projective
coordinate ring of $X$, \[ \Hilb(X; q) = \prod_{k=1}^m [a_k]_q. \]
Here the multiset of degrees $\{a_1, \ldots, a_m\}$ can clearly be
chosen arbitrarily.

For general denominator multisets $\{b_1, \ldots, b_m\}$, such
homogeneous systems of parameters yield complete intersections inside
\textbf{weighted projective space} $\bP^{\{b_1, \ldots, b_m\}}$.
However, it is not at all clear which multisets $\{a_1, \ldots, a_m\}$
are realizable degrees for some homogeneous system of parameters.  The
requirement that $\prod_{k=1}^m [a_k]_q/[b_k]_q \in \bZ_{\geq 0}[q]$
is a significant hurdle, as the next example shows.

\begin{example}
  One may check that
    \[ \frac{[3]_q[5]_q[14]_q}{[2]_q[3]_q[7]_q}
        = 1 + q^2 + q^4 - q^5 + q^6 + q^8 + q^{10} \]
   is not a cyclotomic generating function since it
   has a negative coefficient. Therefore, $\bk[x_1, x_2, x_3]$ with
   $\deg(x_1) = 2, \deg(x_2) = 3, \deg(x_3) = 7$
   has no homogeneous system of parameters with degrees $3, 5, 14$.
   Indeed, in this case, the only monic homogeneous elements
   of degree $3$ and $5$ are $x_2$ and $x_1 x_2$, and one
   may see in a variety of ways (regular sequences, dimension counting,
   nontrivial vanishing) that these two elements cannot belong
   to a homogeneous system of parameters. This example
   also satisfies all of the necessary conditions in
   \Cref{lem:cgf_ineqs}.
\end{example}

\begin{example}
Let $\bk[x_1,...,x_m]$ be a polynomial ring with $\deg(x_k) = b_{k}$
for each $k$. Let $c_1,\ldots , c_m$ be any positive integers.  Then
taking $\theta_{k}=x_{k}^{c_{k}}$ gives a homogeneous system of
parameters whose elements have degrees $a_{k}=c_{k}b_{k}$. Hence \[
\prod_{i=1}^m \frac{[c_kb_k]_q}{[b_k]_q} = \prod_{i=1}^m
[c_k]_{q^{b_k}} \in \Phi^{\alg} \] for all $c_1, \ldots, c_m, b_1,
\ldots, b_m \in \mathbb{Z}_{\geq 1}$.  Therefore, every sequence of
positive integers $b_{1},\ldots, b_{k}$ gives rise to a finite
dimensional quotient ring with Hilbert series given by a CGF with
denominator given by the product of $[b_{i}]_{q}$'s.  
\end{example}

By \Cref{lem:cgf_ineqs}, we know that for any cyclotomic generating
function
\[
f(q)=\prod_{j=1}^m [a_k]_q/[b_k]_q \in \Phi^{+},
\]
including
$f \in \Phi^{\alg}$, we have $a_k \in \Span_{\bZ_{\geq 0}}\{b_1,
\ldots, b_m\}$. It is not immediately clear how to characterize which
sequences $a_1, \ldots, a_k$ can actually be realized as degree
sequences of HSOP's.

\begin{example}
  Take $\deg(x_1) = 2, \deg(x_2) = 3$ in $\bk[x_1, x_2]$.
  Set $\theta_1 = x_1 x_2, \theta_2 = x_1^3 + x_2^2$, so that
  $\deg(\theta_1) = 5, \deg(\theta_2) = 6$. It is easy
  to see that $x_1^4, x_2^4 \in (\theta_1, \theta_2)$,
  so they form an HSOP. The corresponding Hilbert series is
    \[ \frac{[6]_q [5]_q}{[3]_q [2]_q} = q^6 + q^4 + q^3 + q^2 + 1 \in \Phi^{\alg}. \]
\end{example}

\begin{problem}
Besides identifying a specific homogeneous system of parameters in a
graded ring, how can one test membership in the HSOP monoid?
\end{problem}

\begin{problem}
How can one efficiently characterize the minimal
set of generators of $\Phi^{\alg}$?
\end{problem}

\begin{problem}
  Identify the growth rate of $\log |\Phi^{\alg}_n|$ as $n \to
\infty$. 
\end{problem}

One way to explore the HSOP monoid with a fixed denominator multiset
comes from an analogue of \Cref{lem:cgf_graph}.  In particular, the
analogous HSOP subgraph of the graph described in
\Cref{ssec:CGF_denom} remains connected with diameter at most $2m$.

\begin{lemma}\cite[p.7 ``Discussion'']{Hochster.notes.2007}
  Suppose $\bk[x_1, \ldots, x_m]$ is a polynomial ring with homogeneous regular sequences $\theta_1, \ldots, \theta_m$ and
  $\theta_1', \ldots, \theta_m'$. Then there exist homogeneous elements
  $\gamma_1, \ldots, \gamma_m$ such that
    \[ \gamma_1, \ldots, \gamma_i, \theta_{i+1}, \ldots, \theta_m
        \qquad\text{and}\qquad
        \gamma_1, \ldots, \gamma_i, \theta_{i+1}', \ldots, \theta_m' \]
  are regular sequences for all $1 \leq i \leq m$.
\end{lemma}

\begin{remark}
  The argument in \Cref{thm:cgf_tensor} works more generally when
$\bk[x_1, \ldots, x_m]$ is a \textbf{Cohen--Macaulay} $\bN$-graded
$\bk$-algebra of Krull dimension $m$ in the sense of
\cite[\S3]{stanley.1979}, except that the stated expression for
$\Hilb(R; q)$ will be multiplied by some polynomial $P(q) \in \bZ[q]$.
Intuitively, $P(q)$ arises from syzygies of $x_1, \ldots, x_m$ and may
have negative coefficients.  
\end{remark}

The study of Cohen--Macaulay rings has been a rich subject.  It would
be interesting to consider properties of the set of all Hilbert series
of Cohen--Macaulay rings as we have done here for CGFs.  

\begin{problem}
What interesting properties do Hilbert series of Cohen--Macaulay rings
have from the algebraic, probabilistic and analytic perspectives?
\end{problem}

\section{Appendix}\label{s:appendix}

The data below gives the size of $\cM_{n}$ for each monoid discussed
in \Cref{sec:monoid} and \Cref{sub:hsop}.  The sequence
$|\Phi^{\pm}_{n}|$ has a nice generating
function and Vaclav gives an asymptotic approximation in \cite[A120963]{oeis}.  We do not know
of any generating function formulas or asymptotics for the other
sequences, which we added to the OEIS in conjunction with this paper.
Code to produce this data is available at
\begin{center}
\url{https://sites.math.washington.edu/~billey/papers/CGFs/}.
\end{center}

\bigskip

\begin{tiny}
\[
\begin{array}{|l|l|l|}
\hline
\cM & |\cM_{n}| \text{ for } n=1,\ldots ,18 & \text{OEIS}\\
\hline
\Phi^{\lcc} & 1,2,3,5,7,12,16,26,35,53,70,109,142,217,285,418,548,799 & A360622\\
\hline
\Phi^{\uni} &
1,2,3,6,8,14,20,34,48,72,100,162,214,309,437,641,860,1205 & A360621\\
\hline
\Phi^{\gale} &
1,3,4,10,12,27,33,68,82,154,187,346,410,714,857,1460,1722,2860 & A362553\\
\hline
\Phi^{+} &
1,3,4,10,12,27,33,68,82,154,189,350,417,728,874,1492,1767,2937 & A360620\\
\hline
\Phi^{\pm} &
2,6,10,24,38,78,118,224,330,584,838,1420,2002,3258,4514,7134,9754,15010 & A120963\\
\hline
\end{array}
\]
\end{tiny}
\bigskip

The data below gives the number of generators
for each monoid discussed in \Cref{sec:monoid} and \Cref{sub:hsop}.
The sequence of the number of minimal generators of $\Phi^{\pm}$ by degree
in \cite[A014197]{oeis} has a nice
Dirichlet generating function and other formulas.  We do not know of
any generating function formulas or asymptotics for the other
sequences below.

\begin{tiny}
\[
\begin{array}{|l|l|l|}
\hline
\cM & \text{Number of generators of $\cM$ in degree $n$ for } n=1,\ldots ,20\\
\hline \Phi^{\lcc} &
1,1,1,1,1,2,2,4,4,7,8,18,19,37,42,66,87,132,157,252 & A361439
\\
\hline
\Phi^{\uni} &
1,1,1,2,2,3,4,7,10,9,15,28,30,34,66,82,125,126,222,294 & A361440
\\
\hline
\Phi^{\gale} &
1,2,1,3,1,4,1,6,1,5,1,14,2,9,4,28,1,33,14,61 & A362554
\\
\hline
\Phi^{+} &
1,2,1,3,1,4,1,6,1,5,3,16,5,14,6,37,9,46,33,87 & A361441
\\
\hline
\Phi^{\pm} &
2,3,0,4,0,4,0,5,0,2,0,6,0,0,0,6,0,4,0,5 & A014197
\\
\hline
\end{array}
\]
\end{tiny}


\subsection*{Acknowledgments}\label{sec:ack}

We would like to thank Theo Douvropoulos, Darij Grinberg, Sam Hopkins,
Matja\v{z} Konvalinka, Svante Janson, Neil Sloane, Karen Smith, Dennis
Stanton, Richard Stanley, Garcia Sun, Nathan Williams, and Doron
Zeilberger for insightful comments and discussions.  We also thank the
anonymous referees that provided helpful suggestions.

Billey was partially supported by the Washington Research Foundation and
NSF DMS-1764012. Swanson was partially supported by NSF DMS-2348843.



\begin{thebibliography}{CGSHPM22}

  \bibitem[ARW16]{Ardila-Rincon-Williams}
  Federico Ardila, Felipe Rinc\'{o}n, and Lauren Williams.
  \newblock Positroids and non-crossing partitions.
  \newblock {\em Trans. Amer. Math. Soc.}, 368(1):337--363, 2016.
  
  \bibitem[BB05]{b-b}
  Anders Bj\"orner and Francesco Brenti.
  \newblock {\em Combinatorics of Coxeter Groups}, volume 231 of {\em Graduate
    Texts in Mathematics}.
  \newblock Springer, New York, 2005.
  
  \bibitem[Ben73]{Bender}
  Edward~A. Bender.
  \newblock Central and local limit theorems applied to asymptotic enumeration.
  \newblock {\em J. Combinatorial Theory Ser. A}, 15:91--111, 1973.
  
  \bibitem[Bes15]{MR3296817}
  David Bessis.
  \newblock Finite complex reflection arrangements are {$K(\pi,1)$}.
  \newblock {\em Ann. of Math. (2)}, 181(3):809--904, 2015.
  \newblock \doi{10.4007/annals.2015.181.3.1}.
  
  \bibitem[Bil95]{MR1324786}
  Patrick Billingsley.
  \newblock {\em Probability and {M}easure}.
  \newblock Wiley Series in Probability and Mathematical Statistics. John Wiley
    \& Sons, Inc., New York, third edition, 1995.
  \newblock A Wiley-Interscience Publication.
  
  \bibitem[BKS20a]{BKS-asymptotics}
  Sara~C. Billey, Matja\v{z} Konvalinka, and Joshua~P. Swanson.
  \newblock Asymptotic normality of the major index on standard tableaux.
  \newblock {\em Adv. in Appl. Math.}, 113:101972, 36, 2020.
  \newblock \doi{10.1016/j.aam.2019.101972}.
  
  \bibitem[BKS20b]{BKS-zeroes}
  Sara~C. Billey, Matja\v{z} Konvalinka, and Joshua~P. Swanson.
  \newblock Tableau posets and the fake degrees of coinvariant algebras.
  \newblock {\em Adv. Math.}, 371:107252, 46, 2020.
  \newblock \doi{10.1016/j.aim.2020.107252}.
  
  \bibitem[BM52]{BonsallMorris}
  F.~F. Bonsall and Morris Marden.
  \newblock Zeros of self-inversive polynomials.
  \newblock {\em Proc. Amer. Math. Soc.}, 3:471--475, 1952.
  \newblock \doi{10.2307/2031905}.
  
  \bibitem[BM90]{bm88}
  David~W. Boyd and Hugh~L. Montgomery.
  \newblock Cyclotomic partitions.
  \newblock In {\em Number {T}heory ({B}anff, {AB}, 1988)}, pages 7--25. de
    Gruyter, Berlin, 1990.
  
  \bibitem[Br{\"a}15]{Branden2015}
  Petter Br{\"a}nd{\'e}n.
  \newblock Unimodality, log-concavity, real-rootedness and beyond.
  \newblock In {\em Handbook of {E}numerative {C}ombinatorics}, Discrete Math.
    Appl. (Boca Raton), pages 437--483. CRC Press, Boca Raton, FL, 2015.
  
  \bibitem[Bre94]{Brenti1994}
  Francesco Brenti.
  \newblock Log-concave and unimodal sequences in algebra, combinatorics, and
    geometry: an update.
  \newblock In {\em Jerusalem combinatorics '93}, volume 178 of {\em Contemp.
    Math.}, pages 71--89. Amer. Math. Soc., Providence, RI, 1994.
  \newblock \doi{10.1090/conm/178/01893}.
  
  \bibitem[BS22]{BS-asymptotics}
  Sara~C. Billey and Joshua~P. Swanson.
  \newblock The metric space of limit laws for {$q$}-hook formulas.
  \newblock {\em Comb. Theory}, 2(2):Paper No. 5, 58, 2022.
  
  \bibitem[BW89]{MR1022316}
  Anders Bj\"orner and Michelle~L. Wachs.
  \newblock {$q$}-{H}ook length formulas for forests.
  \newblock {\em J. Combin. Theory Ser. A}, 52(2):165--187, 1989.
  \newblock \doi{10.1016/0097-3165(89)90028-9}.
  
  \bibitem[CGSHPM22]{MR4374779}
  Alexandru Ciolan, Pedro~A. Garc\'{\i}a-S\'{a}nchez, Andr\'{e}s Herrera-Poyatos,
    and Pieter Moree.
  \newblock Cyclotomic exponent sequences of numerical semigroups.
  \newblock {\em Discrete Math.}, 345(6):Paper No. 112820, 22, 2022.
  \newblock \doi{10.1016/j.disc.2022.112820}.
  
  \bibitem[CGSM16]{MR3483156}
  Alexandru Ciolan, Pedro~A. Garc\'{\i}a-S\'{a}nchez, and Pieter Moree.
  \newblock Cyclotomic numerical semigroups.
  \newblock {\em SIAM J. Discrete Math.}, 30(2):650--668, 2016.
  \newblock \doi{10.1137/140989479}.
  
  \bibitem[CKW09]{MR2542141}
  Aldo Conca, Christian Krattenthaler, and Junzo Watanabe.
  \newblock Regular sequences of symmetric polynomials.
  \newblock {\em Rend. Semin. Mat. Univ. Padova}, 121:179--199, 2009.
  \newblock \doi{10.4171/RSMUP/121-11}.
  
  \bibitem[CLO15]{Cox-Little-OShea}
  David~A. Cox, John Little, and Donal O'Shea.
  \newblock {\em Ideals, Varieties, and Algorithms}.
  \newblock Undergraduate Texts in Mathematics. Springer, Cham, fourth edition,
    2015.
  \newblock An introduction to computational algebraic geometry and commutative
    algebra.
  \newblock \doi{10.1007/978-3-319-16721-3}.
  
  \bibitem[Coh22]{Cohn22}
  A.~Cohn.
  \newblock \"{U}ber die {A}nzahl der {W}urzeln einer algebraischen {G}leichung
    in einem {K}reise.
  \newblock {\em Math. Z.}, 14(1):110--148, 1922.
  \newblock \doi{10.1007/BF01215894}.
  
  \bibitem[Coh03]{Cohn}
  P.~M. Cohn.
  \newblock {\em Basic Algebra}.
  \newblock Springer-Verlag London, Ltd., London, 2003.
  \newblock Groups, rings and fields.
  \newblock \doi{10.1007/978-0-85729-428-9}.
  
  \bibitem[CWW08]{Chen-Wang-Wang.2008}
  William Y.~C. Chen, Carol~J. Wang, and Larry X.~W. Wang.
  \newblock The limiting distribution of the coefficients of the {$q$}-{C}atalan
    numbers.
  \newblock {\em Proc. Amer. Math. Soc.}, 136(11):3759--3767, 2008.
  \newblock \doi{10.1090/S0002-9939-08-09464-1}.
  
  \bibitem[Dia88]{MR964069}
  Persi Diaconis.
  \newblock {\em Group Representations in Probability and Statistics}, volume~11
    of {\em Institute of Mathematical Statistics Lecture Notes---Monograph
    Series}.
  \newblock Institute of Mathematical Statistics, Hayward, CA, 1988.
  
  \bibitem[Dou18]{MR3940661}
  Theo Douvropoulos.
  \newblock Cyclic sieving for reduced reflection factorizations of the {C}oxeter
    element.
  \newblock {\em S\'{e}m. Lothar. Combin.}, 80B:Art. 86, 12, 2018.
  
  \bibitem[Ful97]{Fulton-book}
  William Fulton.
  \newblock {\em Young Tableaux; With Applications To Representation Theory And
    Geometry}, volume~35 of {\em London Mathematical Society Student Texts}.
  \newblock Cambridge University Press, New York, 1997.
  
  \bibitem[Gal68]{Gale.1968}
  David Gale.
  \newblock Optimal assignments in an ordered set: {A}n application of matroid
    theory.
  \newblock {\em J. Combinatorial Theory}, 4:176--180, 1968.
  
  \bibitem[Gas98]{gasharov97}
  Vesselin Gasharov.
  \newblock Factoring the {P}oincar\'e polynomials for the {B}ruhat order on
    ${S}_{n}$.
  \newblock {\em Combinatorial Theory, Series A}, 83:159--164, 1998.
  
  \bibitem[GR02]{GR2000}
  Vesselin Gasharov and Victor Reiner.
  \newblock Cohomology of smooth {Schubert} varieties in partial flag manifolds.
  \newblock {\em Journal of the {London} Mathematical Society (2)},
    66(3):550--562, 2002.
  
  \bibitem[Hai03]{Haiman.2003}
  Mark Haiman.
  \newblock Combinatorics, symmetric functions, and {H}ilbert schemes.
  \newblock In {\em Current Developments in Mathematics, 2002}, pages 39--111.
    Int. Press, Somerville, MA, 2003.
  
  \bibitem[Har67]{Harper}
  L.~H. Harper.
  \newblock Stirling behavior is asymptotically normal.
  \newblock {\em Ann. Math. Statist.}, 38:410--414, 1967.
  \newblock \doi{10.1214/aoms/1177698956}.
  
  \bibitem[{Hoc}07]{Hochster.notes.2007}
  {Hochster}.
  \newblock Math 711: {L}ecture of {S}eptember 5, 2007.
  \newblock [Online; accessed 23-April-2023].
  \newblock URL:
    \url{https://dept.math.lsa.umich.edu/~hochster/711F07/L09.05.pdf}.
  
  \bibitem[Hop24]{Hopkins.2023}
  Sam Hopkins.
  \newblock Order polynomial product formulas and poset dynamics.
  \newblock {\em Proceedings of Symposia in Pure Mathematics}, 110, 2024.
  
  \bibitem[HST24]{heerten2023probabilistic}
  Nils Heerten, Holger Sambale, and Christoph Th{\"a}le.
  \newblock Probabilistic limit theorems induced by the zeros of polynomials.
  \newblock {\em Math. Nachr.}, 297(5):1772--1792, 2024.
  \newblock \doi{10.1002/mana.202300109}.
  
  \bibitem[HZ15]{HwangZacharovas}
  Hsien-Kuei Hwang and Vytas Zacharovas.
  \newblock Limit distribution of the coefficients of polynomials with only unit
    roots.
  \newblock {\em Random Structures Algorithms}, 46(4):707--738, 2015.
  \newblock \doi{10.1002/rsa.20516}.
  
  \bibitem[IM65]{Iw.M}
  N.~Iwahori and H.~Matsumoto.
  \newblock On some {B}ruhat decomposition and the structure of the {H}ecke rings
    of {$\mathfrak{p}$}-adic {C}hevalley groups.
  \newblock {\em Inst. Hautes \'{E}tudes Sci. Publ. Math.}, 25:5--48, 1965.
  
  \bibitem[Ked08]{Kedlaya}
  Kiran~S. Kedlaya.
  \newblock Search techniques for root-unitary polynomials.
  \newblock In {\em Computational arithmetic geometry}, volume 463 of {\em
    Contemp. Math.}, pages 71--81. Amer. Math. Soc., Providence, RI, 2008.
  \newblock \doi{10.1090/conm/463/09047}.
  
  \bibitem[KKZ11]{MR2775659}
  Christoph Koutschan, Manuel Kauers, and Doron Zeilberger.
  \newblock Proof of {G}eorge {A}ndrews's and {D}avid {R}obbins's {$q$}-{TSPP}
    conjecture.
  \newblock {\em Proc. Natl. Acad. Sci. USA}, 108(6):2196--2199, 2011.
  \newblock \doi{10.1073/pnas.1019186108}.
  
  \bibitem[KM13]{MR3363972}
  Christian Krattenthaler and Thomas~W. M\"{u}ller.
  \newblock Cyclic sieving for generalised non-crossing partitions associated
    with complex reflection groups of exceptional type.
  \newblock In {\em Advances in combinatorics}, pages 209--247. Springer,
    Heidelberg, 2013.
  
  \bibitem[Knu73]{Knuth}
  D.~E. Knuth.
  \newblock {\em {The Art of Computer Programming}}, volume~3.
  \newblock Addison--Wesley, Reading, MA, 1973.
  
  \bibitem[Kro57]{Kronecker.1857}
  L.~Kronecker.
  \newblock Zwei {S}\"{a}tze \"{u}ber {G}leichungen mit ganzzahligen
    {C}oefficienten.
  \newblock {\em J. Reine Angew. Math.}, 53:173--175, 1857.
  \newblock \doi{10.1515/crll.1857.53.173}.
  
  \bibitem[Luk70]{MR0346874}
  Eugene Lukacs.
  \newblock {\em Characteristic Functions}.
  \newblock Hafner Publishing Co., New York, 1970.
  \newblock Second edition.
  
  \bibitem[Mac94]{MR1281612}
  F.~S. Macaulay.
  \newblock {\em The Algebraic Theory of Modular Systems}.
  \newblock Cambridge Mathematical Library. Cambridge University Press,
    Cambridge, 1994.
  \newblock Revised reprint of the 1916 original.
  
  \bibitem[Mat10]{mo:Kronecker}
  MathOverflow.
  \newblock English reference for a result of {K}ronecker?, 2010.
  \newblock [Online; accessed 23-April-2023].
  \newblock URL:
    \url{https://mathoverflow.net/questions/10911/english-reference-for-a-result-of-kronecker}.
  
  \bibitem[MRR82]{MR652647}
  W.~H. Mills, David~P. Robbins, and Howard Rumsey, Jr.
  \newblock Proof of the {M}acdonald conjecture.
  \newblock {\em Invent. Math.}, 66(1):73--87, 1982.
  \newblock \doi{10.1007/BF01404757}.
  
  \bibitem[MS19a]{michelen2019central}
  Marcus Michelen and Julian Sahasrabudhe.
  \newblock Central limit theorems and the geometry of polynomials, 2019.
  \newblock \arxiv{1908.09020}.
  
  \bibitem[MS19b]{Michelen-Sahasrabudhe.2019}
  Marcus Michelen and Julian Sahasrabudhe.
  \newblock Central limit theorems from the roots of probability generating
    functions.
  \newblock {\em Adv. Math.}, 358:106840, 27, 2019.
  \newblock \doi{10.1016/j.aim.2019.106840}.
  
  \bibitem[MV97]{MR1483073}
  Reinhold Meise and Dietmar Vogt.
  \newblock {\em Introduction to Functional Analysis}, volume~2 of {\em Oxford
    Graduate Texts in Mathematics}.
  \newblock The Clarendon Press, Oxford University Press, New York, 1997.
  
  \bibitem[{OEI}23]{oeis}
  {OEIS Foundation Inc.}
  \newblock The {O}n-{L}ine {E}ncyclopedia of {I}nteger {S}equences, 2023.
  \newblock Online. \url{http://oeis.org}.
  
  \bibitem[O'H90]{Ohara.1990}
  Kathleen~M. O'Hara.
  \newblock Unimodality of {G}aussian coefficients: a constructive proof.
  \newblock {\em J. Combin. Theory Ser. A}, 53(1):29--52, 1990.
  \newblock \doi{10.1016/0097-3165(90)90018-R}.
  
  \bibitem[Pit97]{Pitman1997}
  Jim Pitman.
  \newblock Probabilistic bounds on the coefficients of polynomials with only
    real zeros.
  \newblock {\em J. Combin. Theory Ser. A}, 77(2):279--303, 1997.
  \newblock \doi{10.1006/jcta.1997.2747}.
  
  \bibitem[Pro84]{MR782055}
  Robert~A. Proctor.
  \newblock Bruhat lattices, plane partition generating functions, and minuscule
    representations.
  \newblock {\em European J. Combin.}, 5(4):331--350, 1984.
  \newblock \doi{10.1016/S0195-6698(84)80037-2}.
  
  \bibitem[PS19]{MR4038788}
  Robert~A. Proctor and Lindsey~M. Scoppetta.
  \newblock {$d$}-{C}omplete posets: local structural axioms, properties, and
    equivalent definitions.
  \newblock {\em Order}, 36(3):399--422, 2019.
  \newblock \doi{10.1007/s11083-018-9473-4}.
  
  \bibitem[PW99]{Pistone.Wynn.1999}
  Giovanni Pistone and Henry~P. Wynn.
  \newblock Finitely generated cumulants.
  \newblock {\em Statist. Sinica}, 9(4):1029--1052, 1999.
  
  \bibitem[RS18]{MR3884758}
  Victor Reiner and Eric Sommers.
  \newblock Weyl group {$q$}-{K}reweras numbers and cyclic sieving.
  \newblock {\em Ann. Comb.}, 22(4):819--874, 2018.
  
  \bibitem[RSW04]{Reiner-Stanton-White.CSP}
  V.~Reiner, D.~Stanton, and D.~White.
  \newblock The cyclic sieving phenomenon.
  \newblock {\em J. Combin. Theory Ser. A}, 108(1):17--50, 2004.
  \newblock \doi{10.1016/j.jcta.2004.04.009}.
  
  \bibitem[RSW14]{MR3156682}
  Victor Reiner, Dennis Stanton, and Dennis White.
  \newblock What is \dots cyclic sieving?
  \newblock {\em Notices Amer. Math. Soc.}, 61(2):169--171, 2014.
  \newblock \doi{10.1090/noti1084}.
  
  \bibitem[Sac97]{MR1453118}
  Vladimir~N. Sachkov.
  \newblock {\em Probabilistic methods in combinatorial analysis}, volume~56 of
    {\em Encyclopedia of Mathematics and its Applications}.
  \newblock Cambridge University Press, Cambridge, 1997.
  \newblock \doi{10.1017/CBO9780511666193}.
  
  \bibitem[SKKT00]{Smith.etal}
  Karen~E. Smith, Lauri Kahanp\"{a}\"{a}, Pekka Kek\"{a}l\"{a}inen, and William
    Traves.
  \newblock {\em An {I}nvitation to {A}lgebraic {G}eometry}.
  \newblock Universitext. Springer-Verlag, New York, 2000.
  \newblock \doi{10.1007/978-1-4757-4497-2}.
  
  \bibitem[Slo15]{MR3406513}
  William Slofstra.
  \newblock Rationally smooth {S}chubert varieties and inversion hyperplane
    arrangements.
  \newblock {\em Adv. Math.}, 285:709--736, 2015.
  \newblock \doi{10.1016/j.aim.2015.07.034}.
  
  \bibitem[Sta78]{MR0485835}
  Richard~P. Stanley.
  \newblock Hilbert functions of graded algebras.
  \newblock {\em Advances in Math.}, 28(1):57--83, 1978.
  \newblock \doi{10.1016/0001-8708(78)90045-2}.
  
  \bibitem[Sta79]{stanley.1979}
  Richard~P. Stanley.
  \newblock Invariants of finite groups and their applications to combinatorics.
  \newblock {\em Bull. Amer. Math. Soc. (N.S.)}, 1(3):475--511, 1979.
  \newblock \doi{10.1090/S0273-0979-1979-14597-X}.
  
  \bibitem[Sta89]{MR1110850}
  Richard~P. Stanley.
  \newblock Log-concave and unimodal sequences in algebra, combinatorics, and
    geometry.
  \newblock In {\em Graph theory and its applications: {E}ast and {W}est
    ({J}inan, 1986)}, volume 576 of {\em Ann. New York Acad. Sci.}, pages
    500--535. New York Acad. Sci., New York, 1989.
  \newblock \doi{10.1111/j.1749-6632.1989.tb16434.x}.
  
  \bibitem[Sta98]{Stanton.note}
  Dennis Stanton.
  \newblock Fake {G}aussian sequences, Dec. 4, 1998.
  \newblock [Online; accessed 10-May-2023].
  \newblock URL:
    \url{https://www-users.cse.umn.edu/~stant001/PAPERS/macmahon.pdf}.
  
  \bibitem[Sta99]{ec2}
  Richard~P. Stanley.
  \newblock {\em Enumerative {C}ombinatorics. {V}ol. 2}, volume~62 of {\em
    Cambridge Studies in Advanced Mathematics}.
  \newblock Cambridge University Press, Cambridge, 1999.
  \newblock \doi{10.1017/CBO9780511609589}.
  
  \bibitem[Sta12]{ec1}
  Richard~P. Stanley.
  \newblock {\em Enumerative {C}ombinatorics. {V}ol. 1}, volume~49 of {\em
    Cambridge Studies in Advanced Mathematics}.
  \newblock Cambridge University Press, Cambridge, second edition, 2012.
  
  \bibitem[Ste94]{MR1262215}
  John~R. Stembridge.
  \newblock On minuscule representations, plane partitions and involutions in
    complex {L}ie groups.
  \newblock {\em Duke Math. J.}, 73(2):469--490, 1994.
  \newblock \doi{10.1215/S0012-7094-94-07320-1}.
  
  \bibitem[SW98]{MR1692145}
  John~R. Stembridge and Debra~J. Waugh.
  \newblock A {W}eyl group generating function that ought to be better known.
  \newblock {\em Indag. Math. (N.S.)}, 9(3):451--457, 1998.
  \newblock \doi{10.1016/S0019-3577(98)80012-8}.
  
  \bibitem[Syl78]{doi:10.1080/14786447808639408}
  J.J. Sylvester.
  \newblock Proof of the hitherto undemonstrated fundamental theorem of
    invariants.
  \newblock {\em Phil. Mag.}, 5(30):178--188, 1878.
  \newblock \doi{10.1080/14786447808639408}.
  
  \bibitem[{Wik}22]{wiki:cyclotomic}
  {Wikipedia contributors}.
  \newblock Cyclotomic polynomial --- {Wikipedia}{,} the free encyclopedia, 2022.
  \newblock [Online; accessed 23-April-2023].
  \newblock URL:
    \url{https://en.wikipedia.org/w/index.php?title=Cyclotomic_polynomial&oldid=928883454}.
  
  \bibitem[{Wik}23]{enwiki:1185992088}
  {Wikipedia contributors}.
  \newblock Proof of the euler product formula for the riemann zeta function ---
    {Wikipedia}{,} the free encyclopedia.
  \newblock
    \url{https://en.wikipedia.org/w/index.php?title=Proof_of_the_Euler_product_formula_for_the_Riemann_zeta_function&oldid=1185992088},
    2023.
  \newblock [Online; accessed 3-September-2024].
  
  \bibitem[WZ11]{MR2773098}
  S.~Ole Warnaar and W.~Zudilin.
  \newblock A {$q$}-rious positivity.
  \newblock {\em Aequationes Math.}, 81(1-2):177--183, 2011.
  \newblock \doi{10.1007/s00010-010-0055-9}.
  
  \bibitem[Zab03]{math/0310301}
  Mike Zabrocki.
  \newblock A bijective proof of an unusual symmetric group generating function,
    2003.
  \newblock \arxiv{math/0310301}.
  
  \bibitem[Zei89]{Zeilberger.1986}
  Doron Zeilberger.
  \newblock Kathy {O}'{H}ara's constructive proof of the unimodality of the
    {G}aussian polynomials.
  \newblock {\em Amer. Math. Monthly}, 96(7):590--602, 1989.
  \newblock \doi{10.2307/2325177}.
  
  \bibitem[Zei09]{MR2683231}
  Doron Zeilberger.
  \newblock The automatic central limit theorems generator (and much more!).
  \newblock In {\em Advances in combinatorial mathematics}, pages 165--174.
    Springer, Berlin, 2009.
  \newblock \doi{10.1007/978-3-642-03562-3\_8}.
  
  \end{thebibliography}
\end{document}